\documentclass{article}
\usepackage{fullpage}
\usepackage{amsmath,amsfonts,amsthm,amssymb,epsfig}
\usepackage[all]{xy}
\usepackage{enumitem}
\usepackage[bookmarks=true,colorlinks=true, linkcolor=magenta, citecolor=cyan]{hyperref}
\usepackage{graphicx}
\usepackage{float}

\newtheorem{theorem}{Theorem}[section]
\newtheorem{lemma}[theorem]{Lemma}
\newtheorem{proposition}[theorem]{Proposition}
\newtheorem{corollary}[theorem]{Corollary}

\newtheorem{definition}[theorem]{Definition}
\newtheorem{example}[theorem]{Example}

\newtheorem{conjecture}[theorem]{Conjecture}

\begin{document}
\title{p-adic J-homomorphisms and a product formula}
\author{Dustin Clausen}

\maketitle

\begin{abstract}
One-point compactification turns real vector spaces into spheres.  In homotopy theory, this transformation gets encoded in a map called the ``real J-homomorphism".  Here we define and investigate p-adic J-homomorphisms, which sort of turn p-adic vector spaces into spheres.  Our main theorem is a product formula for these J-homomorphisms, saying what happens when you start with a rational vector space.  This formula specializes to Hilbert's version of the quadratic reciprocity law after applying $\pi_2$.
\end{abstract}

\tableofcontents

\section*{Preface}
\subsection{Introduction}

\textit{Notation.}  When $p$ is a prime, $\mathbb{Q}_p$ stands for the $p$-adic numbers.  When $p=\infty$ it stands for the real numbers $\mathbb{R}$.  The phrase $p\leq\infty$ means that $p$ is either a prime number or $\infty$.\\

One-point compactification sends real vector spaces to pointed spheres, in such a way that isomorphisms go to homotopy equivalences and direct sum goes to smash product.  If the resulting spheres are considered only ``stably'' (i.e.\ in the formal limit under iterated suspension), this whole kit and caboodle is encoded in a map of spectra
$$J_\mathbb{R}\colon K(\mathbb{R})\rightarrow Pic(Sp),$$
the real J-homomorphism (or a version of it), from the algebraic K-theory of $\mathbb{R}$ to the spectrum of stable spheres under smash product.

In this paper we produce analogous maps
$$J_{\mathbb{Q}_p}\colon K(\mathbb{Q}_p)_{\geq 2}\rightarrow Pic(Sp)$$
for $p<\infty$, which we call the $p$-adic J-homomorphisms.  The subscript of $\geq 2$ denotes passage to the universal cover; this being necessary means that we don't quite have a way of turning $p$-adic vector spaces into spheres, only almost.

These $p$-adic J-homomorphisms are built in two pieces: an away-from-$p$ piece, based on the ``discrete models map'' for the complex J-homomorphism considered in works of May-Quinn-Ray-Tornehave (\cite{mqrt}) and Snaith (\cite{snaith}); and an at-$p$ piece, based on a construction of Bauer's in the context of $p$-compact groups (\cite{bauer}).

One classical application of the real J-homomorphism, due to Adams (\cite{adams}), concerns the stable homotopy groups of spheres $\pi_*^S$, which make their appearance via canonical isomorphisms $\pi_n^S\simeq \pi_{n+1}Pic(Sp)$ ($n>0$).  Let us ignore $2$-torsion for simplicity.  Then Adams showed that the real J-homomorphism produces a cyclic subgroup of $\pi_{4k-1}^S$ ($k>0$) of order the size of the denominator of the Bernoulli number $B_{2k}/k$, and that (in modern language) these cyclic subgroups completely account for the first chromatic layer of $\pi_*^S$ at odd primes.  We show that the $p$-adic J-homomorphisms have this same property.

\begin{theorem}\label{imj} Let $k>0$. After completion at an odd prime $\ell$, each $J_{\mathbb{Q}_p}$ for $p\leq\infty$ has the same image on $\pi_{4k}$, and this common image maps isomorphically to $\pi_{4k-1}$ of the $E(1)$-local sphere at $\ell$.\end{theorem}

We prove this by producing a neutral candidate for the common image, then identifying it separately with each of the $p$-adic images, simultaneously showing that it maps isomorphically to $\pi_{4k-1}L_{E(1)}S$.

Now we state our main theorem, a product formula relating the different J-homomorphisms.

\begin{theorem}\label{product} The product over all $p\leq\infty$ of the composed maps
$$K(\mathbb{Q})_{\geq 2}\rightarrow K(\mathbb{Q}_p)_{\geq 2}\overset{J_{\mathbb{Q}_p}}{\longrightarrow} Pic(Sp)$$
canonically exists and is canonically trivial (in the sense of Appendix \ref{sum}).\end{theorem}

This reduces to a topological statement, namely a certain refinement of Brouwer degree theory for tori (Theorem \ref{tori}).  However, on $\pi_2$ it recovers a number-theoretic statement, namely Hilbert's product formula for the quadratic norm residue symbols.  Indeed, for each $p\leq\infty$ the map $\pi_2(J_{\mathbb{Q}_p})\colon K_2(\mathbb{Q}_p)\rightarrow \pi^S_1$ is nontrivial\footnote{This does not follow from Theorem \ref{imj}, since it requires $\ell=2$.  But it can be proved either by shadowing the proof of Theorem \ref{imj} and making the necessary adjustments for $\ell=2$, or by explicitly identifying $\pi_1(J_{\mathbb{F}_p})$ (\ref{tame}) with the Legendre symbol through Zolotarev's lemma.  This latter route directly establishes non-triviality only when $p$ is an odd prime, but then Theorem \ref{product} itself can be used to verify the remaining cases $p=2,\infty$ by plugging in suitable elements of $K_2(\mathbb{Q})$.}; but there's only one nontrivial map from $K_2(\mathbb{Q}_p)$ to a group of order two:  the one which, via Matsumoto's presentation of $K_2$ of a field, corresponds to the $p$-adic Hilbert symbol (see \cite{milnor} Appendix).  Thus $\pi_2(J_{\mathbb{Q}_p})$ is an incarnation of the $p$-adic Hilbert symbol, and Theorem \ref{product} refines Hilbert's product formula to the level of spectra.

On higher homotopy groups, the statement given by Theorem \ref{product} is presumably closely related to Banaszak's generalization of Moore's exact sequence to the higher K-groups of number fields (\cite{ban}), when specialized to the number field $\mathbb{Q}$.  However, the connection is not transparent, since our methods and language are different.

\subsection{A correction and apology (please read if you were at Oberwolfach with me in September 2011)}

At Oberwolfach, I made the claim of Theorem \ref{imj} in all degrees, not just degrees a multiple of four.  In other words, I claimed that the relevant images are zero in degrees not a multiple of four.  However, I was mistaken in my argument for the case $p=\ell$.  In fact I don't know how to rule out $J_{\mathbb{Q}_\ell}$ producing higher chromatic classes at $\ell$.  I'm sorry for the misinformation!

\subsection{Outline of contents}\label{outline}

\indent\indent Section \ref{background} lays out some background material, recalled from references.

Section \ref{realj} defines the real J-homomorphism as a map $J_{\mathbb{R}}\colon K(\mathbb{R})\rightarrow Pic(Sp)$, and, using a theorem of Suslin (\cite{suslin}), makes the connection between this $J_{\mathbb{R}}$ and its more classical version.

Section \ref{padicj} defines the $p$-adic J-homomorphisms $J_{\mathbb{Q}_p}\colon K(\mathbb{Q}_p)_{\geq 2}\rightarrow Pic(Sp)$.  These are more complicated than their real analog, and are built in two pieces: the ``tame'' piece, which only sees prime-to-$p$ information, and the ``wild'' piece, which recovers the stuff at $p$.

The tame piece (\ref{tame}) is defined through the K-theory of the residue field $\mathbb{F}_p$, and there combinatorially: a finite dimensional vector space over $\mathbb{F}_p$ is in particular a finite set, and from a finite set we can make a stable map of spheres via the Pontryagin-Thom construction.

The wild piece (\ref{wild}) is defined through the K-theory of the $p$-adic integers $\mathbb{Z}_p$, and there topologically:  a finite free module over $\mathbb{Z}_p$ deloops to a $p$-complete torus, and a $p$-complete torus has a ``stable top cell'' which is a $p$-complete sphere.

Section \ref{imjsect} proves Theorem \ref{imj} (concerning the image of the $J_{\mathbb{Q}_p}$ on homotopy groups).  The key tool here is Rezk's $K(1)$-local logarithm (\cite{log}), which permits a $K(1)$-local analysis of $Pic(Sp)$.

Section \ref{proof} proves the product formula (Theorem \ref{product}), by straightforward reduction to a certain statement about tori.  While this reduced statement is intuitively plausible, its formal proof requires wrangling a bunch of higher homotopical coherences, and we delegate the technical aspects of our chosen formalization to Appendix \ref{spheres}.

Section \ref{remarks} finishes with some remarks and speculations.

Then there are the appendices.  Appendix \ref{sum} discusses the notion of an infinite sum of maps of spectra; this is required for the formulation of Theorem \ref{product}.

Appendix \ref{sdot} recalls Wadhausen's $S_\bullet$-construction from an $\infty$-categorical perspective.  This is the subject of a forthcoming paper of Barwick and Rognes, so we are scant on even the few details we need.

Finally, there is the long Appendix \ref{spheres}, which constitutes the technical heart of the paper.  It shows how to use duality to functorially extract ``stable top cells'' (in the form of invertible spectra) from certain kinds of geometric objects, generalizing e.g.\ part of Rognes's work \cite{stdual} on ``stably dualizable groups''.  To get all the required coherences for this extraction we use a helpful $\infty$-categorical formalism developed by Lurie -- the tensor product on stable presentable $\infty$-categories (\cite{ha} Section 6.3).

\subsection{Thanks}

I've had nice conversations with Dennis Gaitsgory, Mike Hopkins, Thomas Kragh, Anatoly Preygel, Nick Rozenblyum, Vesna Stojanoska, and John Tate concerning this work.  Warm thanks to them for their attention and helpful input.  After multiplication by a sizable factor, these sentiments also extend to my advisor Jacob Lurie, who I further thank for sharing his clarifying insights and perspectives so freely over the past few years.

I'd also like to acknowledge some papers whose influence might not otherwise be apparent from the main text.  The first is that of Hill on the metaplectic groups (\cite{hill}), which also proves Hilbert reciprocity by geometric/homotopical means:  its arguments were inspirational.  Another is the above-referenced paper of Banaszak (\cite{ban}):  its similar prior results provided a valuable source of confidence while ours were being worked out.  Finally, there is the paper of Arbarello, De Concini and Kac (\cite{ack}), with its linear-algebraic proof of Hilbert reciprocity in the function field case:  the original goal for this project was to translate its arguments --- viewed through the lens of algebraic K-theory as advocated by Beilinson, c.f.\ \cite{dr} Section 5.5 --- to the number field case.  In this context, the unfinished paper \cite{ks} of Kapranov-Smirnov was also a source of inspiration, having led to my considering the map $J_{\mathbb{F}_p}$ (Section \ref{tame}).

Finally, I am grateful to the NSF for their support through the GRFP, and to the MIT math department for providing an engaging, tolerant, and friendly workplace.

\section{Background and notation}\label{background}

\subsection{$\infty$-categories}\label{infcat}

We rely on the theory of $(\infty,1)$-categories as developed in Lurie's books \cite{htt} and \cite{ha}.  Following those references, we call $(\infty,1)$-categories simply $\infty$-categories.

Given an $\infty$-category $\mathcal{C}$ and objects $X,Y\in\mathcal{C}$, we denote by $Map_\mathcal{C}(X,Y)$ or just $Map(X,Y)$ the space of morphisms from $X$ to $Y$ in $\mathcal{C}$; it can be viewed as an object in the $\infty$-category $\mathcal{S}$ of spaces (that is, homotopy types, or $\infty$-groupoids).  We further denote by $[X,Y]_\mathcal{C}$ or just $[X,Y]$ the set $\pi_0Map(X,Y)$ (homotopy classes of maps).

A small $\infty$-category $\mathcal{C}$ has a ``space of objects'' $\mathcal{C}^\sim\in\mathcal{S}$, obtained by discarding the non-invertible morphisms from $\mathcal{C}$.  When $\mathcal{C}$ is a symmetric monoidal $\infty$-category, $\mathcal{C}^\sim$ is an $E_\infty$-space (``abelian monoid up to coherent homotopy'').

\subsection{Bousfield localizations of spectra}\label{loc}
Let $Sp$ denote the $\infty$-category of spectra.  We will often want to zoom in on certain phenomena in $Sp$, ignoring others.  Formally, this can be accomplished by means of Bousfield localization.

We define a Bousfield localization of spectra to be an exact accessible localization $L\colon Sp\rightarrow Sp$ (see \cite{htt} Section 5.5.4).  Such an $L$ is uniquely determined by either of two full $\infty$-subcategories of $Sp$:  the essential image of $L$, called the $\infty$-category of $L$-local spectra and denoted $LSp$ (this is the zoomed-in category), or the kernel of $L$, called the $\infty$-category of $L$-acyclic spectra (this is the stuff we're ignoring when we zoom in).  The $\infty$-category $LSp$ is the focus, but often it's the class of $L$-acyclic spectra that's easier to describe.

The functor $L\colon Sp\rightarrow LSp$ makes $LSp$ an idempotent object in the $\infty$-category of stable presentable $\infty$-categories; in particular, $LSp$ carries a canonical symmetric monoidal structure (given on objects by $(X,Y)\mapsto L(X\wedge Y)$), and $L$ canonically promotes to a symmetric monoidal functor $Sp\rightarrow LSp$.  See \cite{ha} Section 6.3.2.

In Bousfield's original paper \cite{bous} on the subject, a convenient method for producing Bousfield localizations is described:  for any spectrum $E$, there is a unique Bousfield localization $L_E\colon Sp\rightarrow Sp$, called the Bousfield localization at $E$, for which the $L_E$-acyclic spectra are exactly those with vanishing $E$-homology.  So the idea here is to zoom in on the phenomena which can be detected by the homology theory $E$.

\subsubsection{Example 1: inversion}\label{inv}

When $E=S\mathbb{Z}[1/P]$ is the $\mathbb{Z}[1/P]$-Moore spectrum for a set $P$ of primes, we denote $L_E(X)$ by $X[1/P]$ and call it $P$-inversion.  Taking homotopy groups $\pi_n$ intertwines this $P$-inversion functor with the analogous one $M\mapsto M[1/P]:=M\otimes\mathbb{Z}[1/P]$ on abelian groups.

When $P$ consists of a single prime $p$, we say $p$-inversion and write $X\mapsto X[1/p]$; when $P$ is all primes but $p$, we say $p$-localization and write $X\mapsto X_{(p)}$.

\subsubsection{Example 2: completion}\label{compl}
When $E=S\mathbb{Z}/p\mathbb{Z}$ is the $\mathbb{Z}/p\mathbb{Z}$-Moore spectrum for a prime $p$, we denote $L_E(X)$ by $X_{\widehat{p}}$ and call it $p$-completion.  On homotopy groups, there is a short exact sequence
$$0\rightarrow Ext(\mathbb{Z}/p^\infty,\pi_nX)\rightarrow\pi_nX_{\widehat{p}}\rightarrow Hom(\mathbb{Z}/p^\infty,\pi_{n-1}X)\rightarrow 0;$$
in particular, when restricted to spectra all of whose homotopy groups have $p$-torsion of finite exponent, taking homotopy groups intertwines this $p$-completion functor with the (somewhat) analogous one $M\mapsto \underset{\leftarrow}{lim}M/p^kM$ on abelian groups.

We call the functor $X\mapsto \prod_p X_{\widehat{p}}$ just completion, and we say $X$ is complete if the natural map $X\rightarrow\prod_p X_{\widehat{p}}$ is an equivalence.  For instance, if $X$ has finite homotopy groups then $X$ is complete.

\subsubsection{Example 3: $K(1)$-localization}\label{k1}
The case $E=K(1)$ of the first Morava K-theory  (at an undenoted prime $\ell$) belongs to the chromatic theory, and is more subtle.  The resulting localization $L_{K(1)}$ can be gotten by first localizing at complex K-theory $K$ and then $\ell$-completing the result.  Practically speaking, $L_{K(1)}$ serves to isolate and amplify ``chromatic level one'' (Adams-Bott-Toda) periodicity phenomena in (mod $\ell^k$) homotopy groups of spectra.  See Ravenel's orange book \cite{orange} for further discussion.

Because of its close connection with periodicity, $K(1)$-localization has some interesting features distinguishing it from inversion and $p$-completion.  For instance, $L_{K(1)}$ kills all bounded-above spectra, and in particular depends only on the connective cover of a spectrum; but in fact an even stronger claim holds, as shown by Bousfield (\cite{bous2}): the functor $L_{K(1)}\colon Sp\rightarrow L_{K(1)}Sp$ factors canonically through the zeroth space functor $\Omega^\infty$ from spectra to pointed spaces, via a functor from pointed spaces to $L_{K(1)}Sp$ called the Bousfield-Kuhn functor (Kuhn did the case of higher $K(n)$).

\subsubsection{Bousfield localizations form a poset}\label{bousfact}
Let $L$ and $L'$ be two Bousfield localizations of spectra, and suppose that every $L$-acyclic spectrum is also $L'$-acyclic.  Then $L'$ canonically factors through $L$, even symmetric monoidally, and we say that $L'$ is a further localization of $L$.  For example, $P'$-inversion is a further localization of $P$-inversion whenever $P'\supseteq P$; also $p$-completion is a further localization of $p$-localization, and $K(1)$-localization is a further localization of both $K$-localization and $\ell$-completion.

Another example is that Bousfield localization at the Eilenberg-Maclane spectrum $H\mathbb{Z}[1/P]$ is a further localization of $P$-inversion.  The resulting natural transformation $(-)[1/P]\rightarrow L_{H\mathbb{Z}[1/p]}$ is an equivalence on bounded below spectra, and thus $L_{H\mathbb{Z}[1/P]}$ is often an acceptable substitute for $P$-inversion.  The same remark applies to $L_{H\mathbb{Z}/p\mathbb{Z}}$ and $p$-completion.

\subsection{Group completion}\label{gp}

Group completion is the homotopy-theoretic analog of the procedure of freely adjoining inverses to a monoid in order to obtain a group.  We discuss it here in the ``weak abelian'' or ``$E_\infty$'' setting.

Let $M$ be an $E_\infty$-space.  Iteratively applying the classifying space construction to $M$ produces a connective spectrum $M^{gp}=B^\infty M$ called the group completion of $M$.  The resulting functor $M\mapsto M^{gp}$ from $E_\infty$-spaces to spectra is left adjoint to the functor sending a spectrum $E$ to its zeroth space $\Omega^\infty E$ with $E_\infty$-structure of loop multiplication.  Furthermore, these adjoint functors $(-)^{gp}$ and $\Omega^\infty$ restrict to an equivalence between group-like $E_\infty$-spaces --- defined by the condition of $\pi_0(-)$ being a group --- and connective spectra.  See Segal's original article \cite{segal}, \cite{gl1} Section 3.5, or \cite{ha} Section 5.1.3 for treatments of these ideas.

\subsection{Picard spectra}\label{pic}
Let $\mathcal{C}$ be a symmetric monoidal $\infty$-category.  Restricting to the invertible components of its $E_\infty$-space of objects $\mathcal{C}^\sim$ (\ref{infcat}) gives a group-like $E_\infty$-space $Inv(\mathcal{C})$.\footnote{One should first verify that $Inv(\mathcal{C})$ is small.  This will indeed be the case in our examples, c.f.\ the following paragraphs.}  We denote the corresponding connective spectrum (\ref{gp}) by $Pic(\mathcal{C}):=Inv(\mathcal{C})^{gp}$, and call it the Picard spectrum of $\mathcal{C}$.

Thus the zeroth homotopy group $\pi_0 Pic(\mathcal{C})$ is the abelian group of equivalence classes of invertible objects of $\mathcal{C}$, and the higher homotopy groups can be accessed via the equivalence $\Omega\Omega^\infty Pic(\mathcal{C})\simeq \Omega Inv(\mathcal{C})\simeq Aut(\mathbf{1})$, the space of self-equivalences of the unit object $\mathbf{1}$ in $\mathcal{C}$.

Most of our examples of interest are of the form $\mathcal{C}=LSp$ for some Bousfield localization $L$ of spectra (\ref{loc}).  For information on $\pi_0Pic(LSp)$ see the paper \cite{hms} of Hopkins, Mahowald, and Sadofsky; as for the higher homotopy groups, we have
$$(\Sigma^{-1}Pic(LSp))_{\geq 0}\simeq (LS)^\times,$$
the spectrum of units of the $E_\infty$-ring spectrum $LS$ (for which see \cite{mqrt} Ch.\ VI or \cite{gl1}); thus $\pi_1 Pic(LSp)$ identifies with the units of $\pi_0 LS$, and for $n\geq 2$ we have $\pi_n Pic(LSp)\simeq  \pi_{n-1} LS$ by translating the unit of $LS$ to zero.

\subsection{Algebraic K-theory}\label{k}
Let $k$ be a unital associative (discrete) ring.  We define the K-theory spectrum $K(k)$ to be the group completion of the $E_\infty$-space $Vect_k^\sim$ of finitely generated projective (left) $k$-modules up to isomorphism under direct sum.

This model for $K(k)$, given by Segal in \cite{segal}, is our primary one; however, for some purposes (such as understanding Quillen's localization theorem -- \cite{Q} Section 5) it is convenient to use a different model, based on the idea of group competing not with respect to the direct sum operation $\oplus$, but instead with respect to the multi-valued operation which says that $B$ is the sum of $A$ and $C$ whenever there's a short exact sequence $0\rightarrow A\rightarrow B\rightarrow C\rightarrow 0$ in $Vect_k$.  This kind of multi-valued group completion can be encoded simplicially using Waldhausen's $S_\bullet$-construction (\cite{wald}); we expose this in Appendix \ref{sdot} in a context that will be convenient for us.

\section{The real J-homomorphism}\label{realj}

In this section, we introduce (or rather recall) the real J-homomorphism, and state a couple of its basic properties.

To every finite-dimensional real vector space we can associate an invertible spectrum, namely the suspension spectrum of its one-point compactification.  This gives an $E_\infty$-map $Vect_{\mathbb{R}}^\sim\rightarrow Inv(Sp)$, direct sum going to smash product; thus on group completion (\ref{gp}) we get a map of spectra
$$J_\mathbb{R}\colon K(\mathbb{R})\rightarrow Pic(Sp)$$
that we call the real J-homomorphism.  Here is a picture:

\begin{figure}[H]
\centering
\includegraphics[scale=.6]{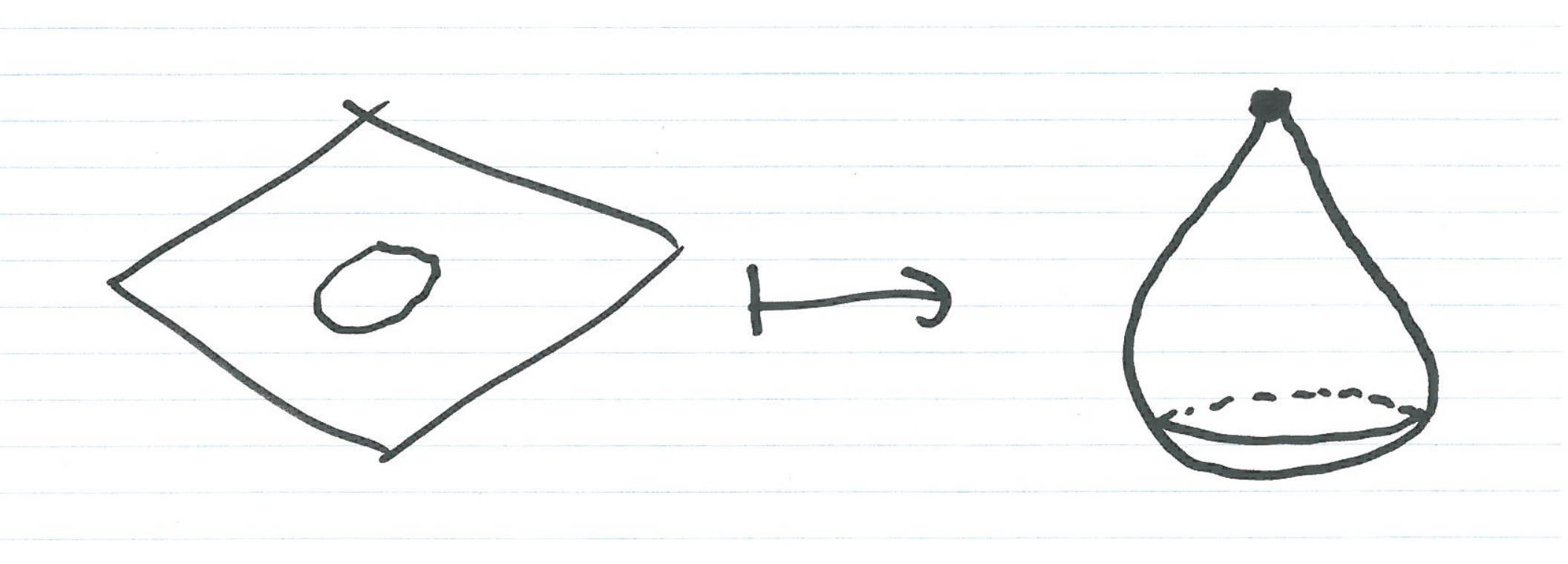}
\caption{$J_\mathbb{R}$ sends the $\mathbb{R}$-vector space on the left to the sphere on the right.}
\end{figure}

\textit{Remark.}  Here is the relationship between $J_{\mathbb{R}}$ and the stable real J-homomorphism as classically defined.  Since the one-point compactification functor on finite-dimensional real vector spaces $V$ is continuous for the Euclidean topology on $GL(V)$, we see that $J_\mathbb{R}$ factors naturally through $K^{top}(\mathbb{R})\simeq\mathbb{Z}\times BO$, this being the group completion of $Vect_\mathbb{R}^\sim$ viewed as a topological category:
$$K(\mathbb{R})\rightarrow K^{top}(\mathbb{R}) \rightarrow Pic(Sp).$$
On $1$-connected covers, the second map is equivalent to the classical stable real J-homomorphism $SO\rightarrow SG$ from the infinite special orthogonal group to the infinite loop space of degree one self-maps of the sphere spectrum.  Thus our $J_\mathbb{R}$ can be recovered from the classical one.  But the converse is true as well, since on the one hand the map $K(\mathbb{R})\rightarrow K^{top}(\mathbb{R})$ is an equivalence on completion (\ref{compl}) by a result of Suslin (\cite{suslin} Cor.\ 4.7), and on the other hand $Pic(Sp)_{\geq 1}$ is complete by Serre finiteness.\\

We will need the following trivial lemma concerning the behavior of $J_\mathbb{R}$ on $\pi_0$:

\begin{lemma}\label{lowreal}
On $\pi_0$, the map $J_{\mathbb{R}}$ induces the identity $\mathbb{Z}\rightarrow\mathbb{Z}$ (i.e.\ it sends the unit vector space $\mathbb{R}$ to the one-sphere $\Sigma S$).
\end{lemma}

More seriously, we also need the following lemma concerning the map $J_\mathbb{C}\colon K(\mathbb{C})\rightarrow Pic(Sp)$ given by precomposing $J_\mathbb{R}$ with the forgetful map $K(\mathbb{C})\rightarrow K(\mathbb{R})$.  It is essentially part of Sullivan's proof of the complex Adams conjecture (\cite{sullivan}); we include it here only in order for our proof of Theorem \ref{imj} to be self-contained in matters of $J$.

\begin{lemma}\label{discreteadams}
Let $\ell$ be any prime.  Then the composition $K(\mathbb{C})\rightarrow Pic(Sp)\rightarrow Pic(Sp_{\widehat{\ell}})$ of $J_\mathbb{C}$ with the $\ell$-completion map is (homotopy) invariant under the action of $Aut(\mathbb{C})$ on $K(\mathbb{C})$.
\end{lemma}

Here $Aut(\mathbb{C})$ denotes the automorphism group of the abstract field $\mathbb{C}$.

\begin{proof}
Let $Var_\mathbb{C}$ denote the symmetric monoidal category of varieties over $\mathbb{C}$ under cartesian product, and let $h$ denote the symmetric monoidal functor $Var_\mathbb{C}\rightarrow Sp_{\widehat{\ell}}$ given by $X\mapsto (\Sigma^\infty_+X(\mathbb{C}))_{\widehat{\ell}}$.  Then \'{e}tale homotopy theory shows that $h$ is actually functorial for all maps of schemes, and is therefore invariant under the action of $Aut(\mathbb{C})$ on $Var_\mathbb{C}$.  But we can $Aut(\mathbb{C})$-equivariantly factor $(J_{\mathbb{C}})_{\widehat{\ell}}$ through $h$, up to colimits:  indeed, the one-point compactification of a $V\in Vect_\mathbb{C}^\sim$ identifies with the homotopy cofiber of $(V-0)\rightarrow V$; but both $V-0$ and $V$ canonically promote to varieties over $\mathbb{C}$, letting us interpret $(J_{\mathbb{C}})_{\widehat{\ell}}$ as the group completion of the $E_\infty$-map $Vect_\mathbb{C}^\sim\rightarrow Inv(Sp_{\widehat{\ell}})$ given by $V\mapsto cofib(h(V-0)\rightarrow h(V))$.
\end{proof}

\section{The $p$-adic J-homomorphism}\label{padicj}

Let $p$ be a prime.  In this section we introduce the $p$-adic J-homomorphisms, and calculate their effects on low homotopy groups.

The localization functors $Sp\rightarrow Sp_{\widehat{p}}$ and $Sp\rightarrow Sp[1/p]$ (\ref{loc}) induce on Picard spectra (\ref{pic}) a map
$$Pic(Sp)\rightarrow Pic(Sp_{\widehat{p}})\times Pic(Sp[1/p]).$$
This is an isomorphism on $\pi_i$ for $i\geq 2$, and hence an equivalence on $1$-connected covers.  Thus we may define our $p$-adic J-homomorphism $J_{\mathbb{Q}_p}\colon K(\mathbb{Q}_p)_{\geq 2}\rightarrow Pic(Sp)$ in two separate pieces: a tame piece $J_{\mathbb{Q}_p}^{tame}\colon K(\mathbb{Q}_p)_{\geq 2}\rightarrow Pic(Sp[1/p])$ and a wild piece $J_{\mathbb{Q}_p}^{wild}\colon K(\mathbb{Q}_p)_{\geq 2}\rightarrow Pic(Sp_{\widehat{p}})$.

Crucial to the definition of these pieces will be the ``localization'' fiber sequence
$$K(\mathbb{F}_p)\rightarrow K(\mathbb{Z}_p)\rightarrow K(\mathbb{Q}_p)$$
of \cite{Q} Section 5, which by rotation to $K(\mathbb{Z}_p)\rightarrow K(\mathbb{Q}_p)\overset{\partial}{\longrightarrow} \Sigma K(\mathbb{F}_p)$ provides a (partial) decomposition of $K(\mathbb{Q}_p)$ to match the above decomposition of $Pic(Sp)$.

\subsection{The tame $p$-adic J-homomorphism}\label{tame}

The tame piece is based on a certain map
$$J_{\mathbb{F}_p}\colon K(\mathbb{F}_p)\rightarrow (S[1/p])^\times$$
from the K-theory of $\mathbb{F}_p$ to the units of the $p$-inverted sphere, in the following way: $J_{\mathbb{Q}_p}^{tame}$ is the composition
$$K(\mathbb{Q}_p)\overset{\partial}{\longrightarrow} \Sigma K(\mathbb{F}_p)\overset{\Sigma J_{\mathbb{F}_p}}{\longrightarrow} \Sigma (S[1/p])^\times\rightarrow Pic(Sp[1/p]),$$
where the last map is the 0-connected cover.  (Thus $J_{\mathbb{Q}_p}^{tame}$ exists on the whole $K(\mathbb{Q}_p)$, not just on $K(\mathbb{Q}_p)_{\geq 2}$.)\\

In turn, $J_{\mathbb{F}_p}$ is defined as the group completion of the $E_\infty$-map $Vect_{\mathbb{F}_p}^\sim\rightarrow Aut(S[1/p])$ (direct sum going to smash product) given as the composition
$$Vect_{\mathbb{F}_p}^\sim\rightarrow Set_{p}^\sim\rightarrow Map_{p}(S,S)\rightarrow Aut(S[1/p]),$$
where:

\begin{enumerate}
\item The first map forgets down to the $E_\infty$-space of finite sets of $p$-power order under cartesian product;
\item The second map comes from the fact that a finite set canonically determines a stable self-map of the $0$-sphere of degree the cardinality of that set, with cartesian product of sets going to smash product of maps;
\item And the last map is induced by $p$-inversion, which makes all $p$-power-degree maps of spheres invertible.
\end{enumerate}

Here is a picture:

\begin{figure}[H]
\centering
\includegraphics[scale=.55]{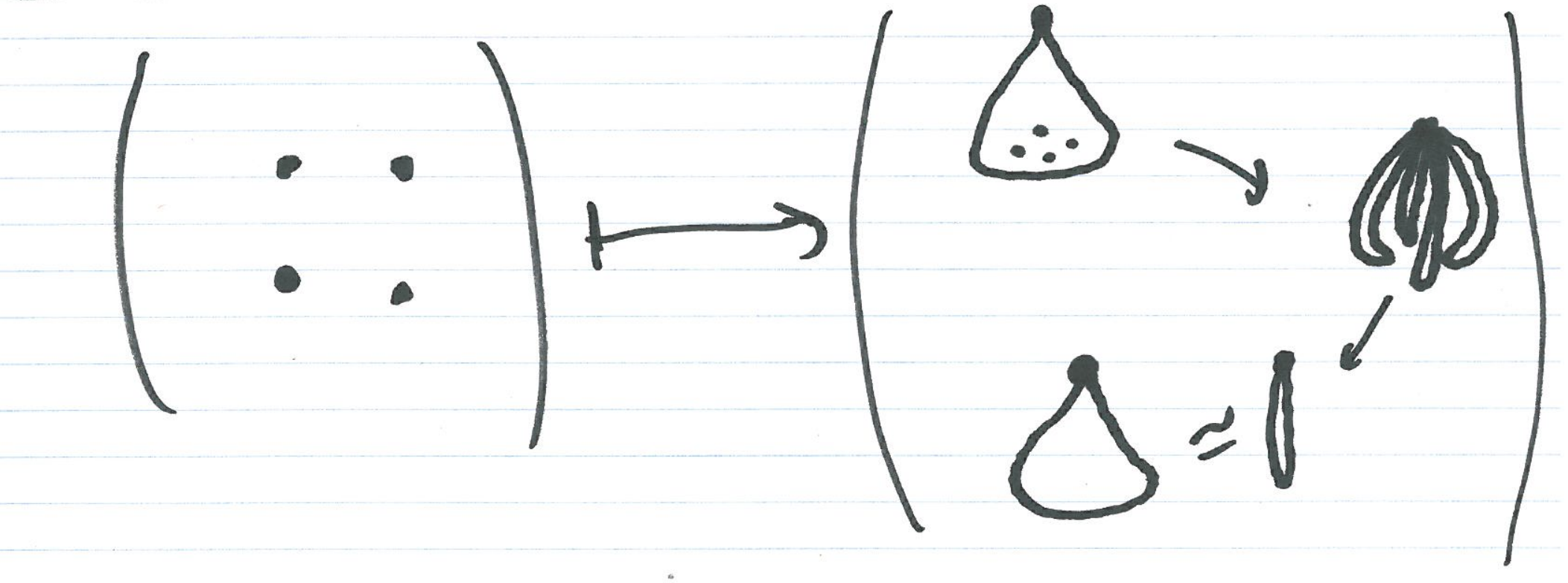}
\caption{$J_{\mathbb{F}_2}$ sends the $\mathbb{F}_2$-vector space on the left to the map of spheres on the right.}
\end{figure}

The second $E_\infty$-map $Set_p^\sim\rightarrow Map_p(S,S)$ above, though readily visualized using the Pontryagin-Thom construction (as in the preceding picture), is potentially opaque from a rigorous standpoint.  Thus let us indicate here how it can be formalized by means of the tensor product operation on presentable $\infty$-categories (\cite{ha} Section 6.3), via the discussion of sheaves of spectra in Appendix \ref{spheres}:

Given $F\in Set_p$, let $f\colon F\rightarrow *$ denote the projection to the point and $f^*\colon Sh(*)\rightarrow Sh(F)$ the resulting pullback functor for sheaves of spectra (\ref{orb2shf}).  Then our desired map $Set_p^\sim\rightarrow Map_p(S,S)$ can be interpreted as sending $F$ to the composition
$$S\rightarrow f_*f^*S\simeq f_\natural f^* S\rightarrow S,$$
where the first map is the unit for the $(f^*,f_*)$ adjunction, the last map is the counit for the $(f_\natural,f^*)$ adjunction, and the middle map comes from a canonical equivalence $f_\natural\simeq f_*$ (encoding the fact that $F$-indexed coproducts and products agree for spectra) which is produced as follows:  the diagonal $\Delta\colon F\rightarrow F\times F$ is a monomorphism, so its dualizing sheaf (\ref{dualdef}) is canonically trivial, giving a canonical equivalence $\Delta_*\simeq \Delta_\natural$ (\ref{dualprop}), which then canonically trivializes the dualizing sheaf of $F$ and hence gives the desired equivalence $f_\natural\simeq f_*$.

The point of this rigamarole is that now Lemma \ref{shprod} and Propostion \ref{cohprop} automatically give the required $E_\infty$-functoriality, since all of the adjoint functors used above preserve colimits.\\

\textit{Remark.} This $J_{\mathbb{F}_p}$ goes by the name ``discrete models map'' in the 1970's topology literature:  ``discrete'' because it is essentially recovered as the group completion of the purely combinatorial first map $Vect_{\mathbb{F}_p}^\sim\rightarrow Set_p^\sim$ in the above sequence; and ``model'' because, granting the (verified) Adams conjecture, it can be used to model the $\ell$-primary complex J-homomorphism whenever $p$ generates $\mathbb{Z}_\ell^\times$.  See Chapter XVIII of the book \cite{mqrt}, Snaith's article \cite{snaith}, and  \cite{ms} Example (iii).\\

We finish with a lemma concerning the values of $J_{\mathbb{Q}_p}^{tame}$ on low homotopy groups:

\begin{lemma}\label{lowtame}
The map $J_{\mathbb{Q}_p}^{tame}\colon K(\mathbb{Q}_p)\rightarrow Pic(Sp[1/p])$ has the following properties:
\begin{enumerate}
\item On $\pi_0$ it is zero;
\item On $\pi_1$ it identifies with the homomorphism $\mathbb{Q}_p^\times\rightarrow \mathbb{Z}[1/p]^\times$ given by $x\mapsto \|x\|_p^{-1}$.
\end{enumerate}
\end{lemma}
\begin{proof}
By \cite{Q} Lemma 5.16, the boundary map $\partial\colon K(\mathbb{Q}_p)\rightarrow\Sigma K(\mathbb{F}_p)$ is already zero on $\pi_0$, and on $\pi_1$ induces the homomorphism $\mathbb{Q}_p^\times\rightarrow\mathbb{Z}$ given by $x\mapsto v_p(x)$; on the other hand $J_{\mathbb{F}_p}\colon K(\mathbb{F}_p)\rightarrow (S[1/p])^\times$ induces $k\mapsto p^k$ on $\pi_0$, since a $k$-dimensional vector space over $\mathbb{F}_p$ has cardinality $p^k$ and therefore induces a degree-$p^k$ self-map of $S$ via the above construction.
\end{proof}
\subsection{The wild $p$-adic J-homomorphism}\label{wild}

The wild piece is based on a certain map
$$J_{\mathbb{Z}_p}\colon K(\mathbb{Z}_p)\rightarrow Pic(Sp_{\widehat{p}})$$
from the K-theory of the $p$-adic integers to $Pic(Sp_{\widehat{p}})$, in the following way:  on the one hand, the map $K(\mathbb{Z}_p)_{\geq 2}\rightarrow K(\mathbb{Q}_p)_{\geq 2}$ is a $p$-local equivalence, since in the localization sequence $K(\mathbb{F}_p)\rightarrow K(\mathbb{Z}_p)\rightarrow K(\mathbb{Q}_p)$ the fiber $K(\mathbb{F}_p)$ is prime-to-$p$ in positive degrees by Quillen's calculation (\cite{quillenfinite}) of the groups $K_*(\mathbb{F}_p)$\footnote{The full calculation is not necessary here:  the material in the self-contained Section 11 of loc.\ cit.\ suffices.}; but on the other hand, the spectrum $Pic(Sp_{\widehat{p}})_{\geq 2}$ is $p$-local.  Thus the map $J_{\mathbb{Z}_p}$ restricted to $K(\mathbb{Z}_p)_{\geq 2}$ essentially uniquely extends to the desired map $J_{\mathbb{Q}_p}^{wild}\colon K(\mathbb{Q}_p)_{\geq 2}\rightarrow Pic(Sp_{\widehat{p}})$.\\

In turn, $J_{\mathbb{Z}_p}$ is defined to be the group completion of the $E_\infty$-map $Vect_{\mathbb{Z}_p}^\sim\rightarrow Inv(Sp_{\widehat{p}})$ (direct sum going to $p$-complete smash product) defined as follows:  given a finite free $\mathbb{Z}_p$-module $M$, we take its classifying space $BM$, viewing $M$ merely as a discrete group; this $BM$ is a sort of $p$-complete torus, and so it has a ``stable top cell'' which is a $p$-complete sphere, and we send $M$ to the inverse of this $p$-complete sphere viewed as an invertible object in $Sp_{\widehat{p}}$.  Here is a picture:

\begin{figure}[H]
\centering
\includegraphics[scale=.7]{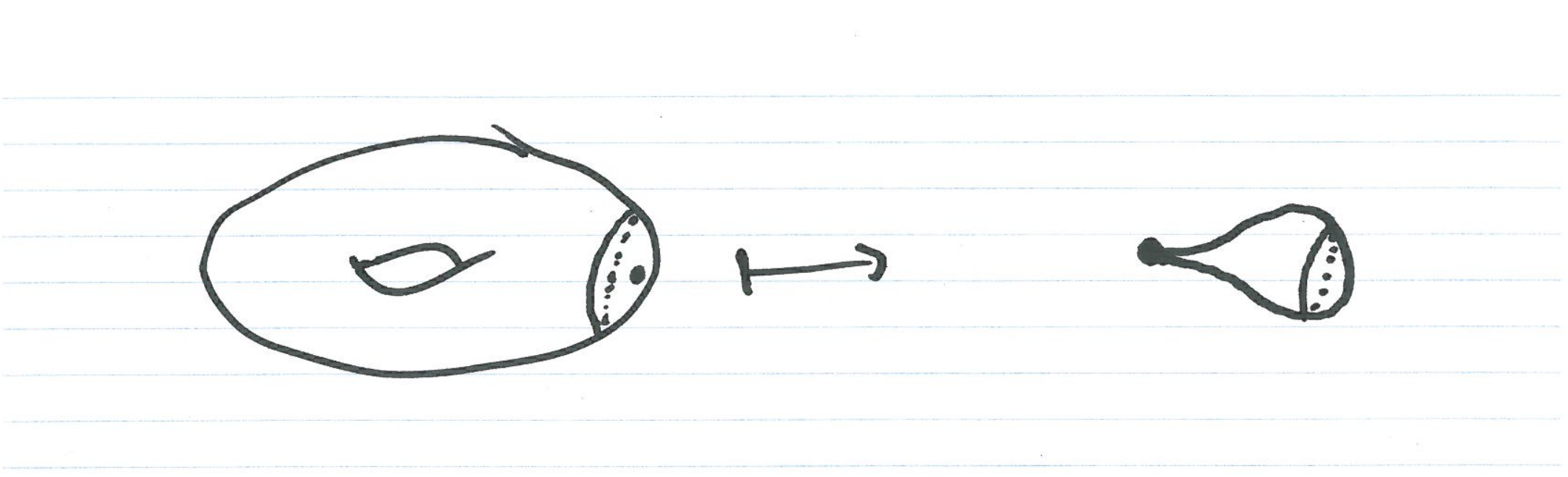}
\caption{$J_{\mathbb{Z}_p}$ sends $H_1(-;\mathbb{Z}_p)$ of the torus on the left to the inverse of the sphere on the right.  You have to imagine the torus and sphere $p$-completed.}
\end{figure}

We can formalize this using the notion of Poincar\'{e} duality $E_\infty$-algebra described in Appendix \ref{spheres}.  Indeed, we claim that for $M\in Vect_{\mathbb{Z}_p}^\sim$, the spectrum $C_*(BM):=(\Sigma^\infty_+BM)_{\widehat{p}}$ is dualizable in $Sp_{\widehat{p}}$, and its would-be-dual $C^*(BM)$ --- the spectrum of maps from $BM$ to the $p$-complete sphere --- is a Poincar\'{e} duality $E_\infty$-algebra in $Sp_{\widehat{p}}$ (Definition \ref{pdalg}), meaning $C_*(BM)$ is furthermore invertible as a module over $C^*(BM)$.  Indeed, if $L$ denotes a finite free $\mathbb{Z}$-module with $L\otimes_\mathbb{Z}\mathbb{Z}_p\simeq M$, then the map $\Sigma^\infty_+BL\rightarrow\Sigma^\infty_+BM$ is a (mod $p$) homology isomorphism between connective spectra and hence an equivalence after $p$-completion, so we are reduced to checking the analogous claims for $C_*(BL)$ and $C^*(BL)$.  But $BL$ is homotopy equivalent to a compact parallelizable manifold, namely a torus, so this follows from part 2 of Theorem \ref{unithm}.

Thus we can formally describe the $E_\infty$-map $Vect_{\mathbb{Z}_p}\rightarrow Inv(Sp_{\widehat{p}})$ whose group completion gives $J_{\mathbb{Z}_p}$ as sending $M$ to the inverse of the $p$-complete sphere
$$Sph^{alg}(C^*(BM)):=C_*(BM)\otimes_{C^*(BM)}S_{\widehat{p}}$$
of Theorem \ref{algsphere}, where the augmentation $C^*(BM)\rightarrow S_{\widehat{p}}$ comes from the canonical point of $BM$.\\

\textit{Remark.} This association of a $p$-complete sphere to a $p$-complete torus was originally considered by Bauer (\cite{bauer}), in the general context of $p$-compact groups.  Note, however, that our method for formalizing this association differs from Bauer's, for instance in that here the torus only needs to be pointed, whereas in \cite{bauer} its full group structure is used.\\

We finish by determining what $J_{\mathbb{Z}_p}$ does on low homotopy groups.

\begin{lemma}\label{lowwild}
The map $J_{\mathbb{Z}_p}\colon K(\mathbb{Z}_p)\rightarrow Pic(Sp_{\widehat{p}})$ has the following properties:
\begin{enumerate}
\item On $\pi_0$, it induces the map $\mathbb{Z}\rightarrow\mathbb{Z}$ given by $k\mapsto -k$;
\item On $\pi_1$, it induces the map $\mathbb{Z}_p^\times\rightarrow\mathbb{Z}_p^\times$ given by $x\mapsto x^{-1}$.
\end{enumerate}
\end{lemma}

\begin{proof}
For a $p$-complete spectrum $X$ and an integer $k$, we denote by $H_k(X;\mathbb{Z}_p)$ the $\mathbb{Z}_p$-module $\pi_k(X\wedge H\mathbb{Z}_p)$, using the $p$-complete smash product.  Then, since $\pi_0 Pic(Sp_{\widehat{p}})\simeq\mathbb{Z}$ just says how shifted a sphere is and $\pi_1 Pic(Sp_{\widehat{p}})\simeq\mathbb{Z}_p^\times$ just says what degree a self-equivalence of a sphere has, both claims can be checked on the level of this $p$-adic homology, and more specifically it suffices to show that the $p$-complete spectrum $X:=J_{\mathbb{Z}_p}([\mathbb{Z}_p])^{-1}=C_*(B\mathbb{Z}_p)\otimes_{C^*(B\mathbb{Z}_p)}S_{\widehat{p}}$ satisfies:
\begin{enumerate}
\item $H_k(X;\mathbb{Z}_p)=0$ for $k\neq 1$;
\item The action of $\mathbb{Z}_p^\times$ on $B\mathbb{Z}_p$ induces the action of $\mathbb{Z}_p^\times$ by scalars on $H_1(X;\mathbb{Z}_p)$.
\end{enumerate}
However, point 1 is clear by comparison with $B\mathbb{Z}=S^1$, and for point 2 we note that the natural map $C_*(B\mathbb{Z}_p)\rightarrow X$ is an isomorphism on $H_1(-;\mathbb{Z}_p)$, again by comparison with $B\mathbb{Z}$.
\end{proof}

\section{The image of J}\label{imjsect}
In this section, predominately calculational, we prove Theorem \ref{imj} (concerning the image of the $J_{\mathbb{Q}_p}$ on homotopy groups).  The elements of $\pi_*^S$ appearing in Theorem \ref{imj} are detected at the first chromatic level, so we start by recalling some $K(1)$-local preliminaries.

For the rest of this section $\ell$ will denote a fixed odd prime, and our $K(1)$-localizations (\ref{k1}) will be implicitly taken at $\ell$.  The reason for assuming $\ell$ odd is that in this case the group of $\ell$-adic units $\mathbb{Z}_\ell^\times$ is procyclic, which, as we will remind, implies a simplified description of $L_{K(1)}S$.

\subsection{$\ell$-completed complex K-theory}

The governing object of $L_{K(1)}Sp$ is the $\ell$-completed complex K-theory spectrum $K_{\widehat{\ell}}\in L_{K(1)}Sp$.  Recall that $K_{\widehat{\ell}}$ is canonically an $E_\infty$-algebra in $L_{K(1)}Sp$, and moreover carries, for each $u\in\mathbb{Z}_\ell^\times$, a unique $E_\infty$-algebra automorphism $\psi^u$ (the $u^{th}$-power Adams operation) inducing multiplication by $u$ on $\pi_2(K_{\widehat{\ell}})$ (see \cite{gh} Cor.\ 7.7).  Here is a picture:

\begin{figure}[H]
\centering
\includegraphics[scale=.55]{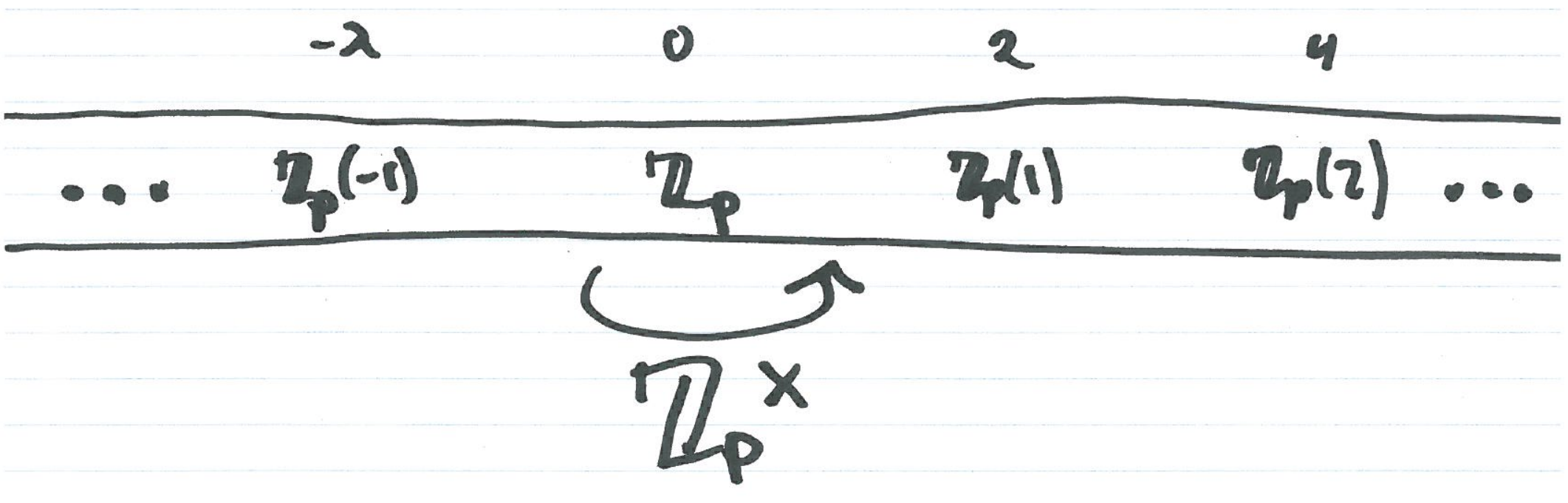}
\caption{A picture of $K_{\widehat{p}}$, or at least of its homotopy groups.  Sorry for drawing $p$ instead of $\ell$.}
\end{figure}

For our purposes, we will produce this structure on $K_{\widehat{\ell}}$ as follows.  By Suslin's theorem (\cite{suslin} Cor.\ 4.7), the natural map $K(\mathbb{C})\rightarrow K^{top}(\mathbb{C})$ is an equivalence on $\ell$-completion; on the other hand, if $\beta\colon K^{top}(\mathbb{C})\rightarrow \Sigma^{-2}K^{top}(\mathbb{C})$ denotes the Bott map, then the $\beta$-inversion $K^{top}(\mathbb{C})\rightarrow K^{top}(\mathbb{C})[\beta^{-1}]\simeq K$ realizes $K^{top}(\mathbb{C})$ as the connective cover of $K$.  Thus in total we have $L_{K(1)}K(\mathbb{C})\simeq K_{\widehat{\ell}}$, and it suffices to equip $K(\mathbb{C})$ with a canonical $E_\infty$-ring structure admitting an automorphism that induces $\zeta\mapsto\zeta^u$ on all $\ell$-power-torsion elements $\zeta\in \pi_1K(\mathbb{C})\simeq\mathbb{C}^\times$ (since $\pi_2K(\mathbb{C})_{\widehat{\ell}}\simeq Hom(\mathbb{Z}/\ell^\infty,\pi_1K(\mathbb{C}))$, c.f.\ \ref{compl}).

However, the symmetric monoidal tensor product on $Vect_\mathbb{C}$ induces the desired $E_\infty$-ring structure on $K(\mathbb{C})$, and we can then produce an automorphism of $K(\mathbb{C})$ of the desired sort by choosing an automorphism of the discrete field $\mathbb{C}$ which induces $\zeta\mapsto\zeta^{u^{-1}}$ on any $\ell$-power root of unity $\zeta\in\mathbb{C}^\times$.

\subsection{The $K(1)$-local sphere}\label{k1sph}

Fix a generator $u\in\mathbb{Z}_\ell^\times$.  Then the main calculation is that the unit map $L_{K(1)}S\rightarrow K_{\widehat{\ell}}$ induces a fiber sequence of spectra
$$L_{K(1)}S\rightarrow K_{\widehat{\ell}}\overset{\psi^u-1}{\longrightarrow}K_{\widehat{\ell}}.$$
(In canonical terms we have an equivalence of $L_{K(1)}S$ with the homotopy fixed points of a profinite action of $\mathbb{Z}_\ell^\times$ on $K_{\widehat{\ell}}$ by Adams operations --- see \cite{dh} and \cite{bd} --- but given a generator $u\in\mathbb{Z}_\ell^\times$ this simplifies to the above statement, which is anyway classical, c.f.\ \cite{bous} Section 4.)

This fiber sequence permits a detailed analysis of $L_{K(1)}S$; for instance it allows to calculate the homotopy groups of $L_{K(1)}S$:
\begin{enumerate}
\item $\pi_0L_{K(1)}S$ identifies with $\mathbb{Z}_\ell$.
\item $\pi_{2k-1}L_{K(1)}S$ identifies with $\mathbb{Z}_\ell/(u^k-1)\mathbb{Z}_\ell$ --- canonically, $H^1(\mathbb{Z}_\ell^\times,\mathbb{Z}_\ell(k))$ --- for all $k\in\mathbb{Z}$;
\item $\pi_n L_{K(1)}S=0$ for all other $n$ (i.e.\ nonzero even integers $n$).
\end{enumerate}

Note that $\mathbb{Z}_\ell/(u^k-1)\mathbb{Z}_\ell$ is finite except when $k=0$, where we find that $\pi_{-1}L_{K(1)}S$ identifies with $\mathbb{Z}_\ell$ (canonically, $Hom(\mathbb{Z}_\ell^\times,\mathbb{Z}_\ell)$).

\subsection{The $K(1)$-local logarithm}

We recall the following notion of logarithm for $K(1)$-local $E_\infty$-ring spectra, studied by Rezk in \cite{log}.  Let $A$ be an $E_\infty$-algebra in $L_{K(1)}Sp$, and let $A^\times$ denote the spectrum of units of $A$.  Then translating the unit $1\in\Omega^\infty A$ to zero $0\in\Omega^\infty A$ supplies an equivalence of pointed spaces $\Omega^\infty\Sigma^{-1}A^\times\simeq\Omega^\infty\Sigma^{-1}A$, so the Bousfield-Kuhn functor (\ref{k1}) gives $L_{K(1)}A^\times\simeq L_{K(1)}A\simeq A$, whence a natural map of spectra
$$log_A\colon A^\times\rightarrow A,$$
functorial for maps of $E_\infty$-algebras in $L_{K(1)}Sp$.  Rezk gives a formula for the evaluation of $log_A$ on any finite complex $X$ in terms of $\theta$-operations (\cite{log} Thm.\ 1.9); when $A=K_{\widehat{\ell}}$ and $X=*$, we find:
\begin{proposition}\label{rzk}
The map $log_{K_{\widehat{\ell}}}\colon K_{\widehat{\ell}}^\times\rightarrow K_{\widehat{\ell}}$ is surjective on $\pi_0$.
\end{proposition}
\begin{proof} Indeed, Rezk's formula shows that $\pi_0(log_{K_{\widehat{\ell}}})\colon \mathbb{Z}_\ell^\times\rightarrow\mathbb{Z}_\ell$ is given by $x\mapsto \frac{1}{\ell}log(x^{\ell-1})$.\end{proof}

\subsection{A $K(1)$-local analysis of $Pic(Sp)$}

Now we can understand $Pic(Sp)$ well enough in order to prove Theorem \ref{imj}.  Here is the key proposition:

\begin{proposition}\label{k1analysis}
There is a map $log\colon Pic(L_{K(1)}Sp)\rightarrow \Sigma L_{K(1)}S$ with the following properties:
\begin{enumerate}
\item $\pi_1 log$ sends any generator of $\pi_1Pic(L_{K(1)}Sp)\simeq\mathbb{Z}_\ell^\times$ to a generator of $\pi_0L_{K(1)}S\simeq \mathbb{Z}_\ell$;
\item $\pi_0 log$ sends the class of the $2$-sphere to a generator of $\pi_{-1}L_{K(1)}S\simeq Hom(\mathbb{Z}_\ell^\times,\mathbb{Z}_\ell)$.
\end{enumerate}
\end{proposition}
In fact, we will see that $\pi_1log$ sends $exp(\frac{\ell}{\ell-1})\in\mathbb{Z}_\ell^\times$ to $1\in\mathbb{Z}_\ell$, and $\pi_0log$ sends the $2$-sphere to the generator $x\mapsto \frac{1}{\ell}log(x^{\ell-1})$ of $Hom(\mathbb{Z}_\ell^\times,\mathbb{Z}_\ell)$.  

\begin{proof}
Denote by $Pic(K_{\widehat{\ell}})$ the Picard spectrum (\ref{pic}) of the symmetric monoidal $\infty$-category $Mod_{K_{\widehat{\ell}}}$ of $K_{\widehat{\ell}}$-modules in $L_{K(1)}Sp$.  Then the natural map $K_{\widehat{\ell}}^\times\rightarrow \Sigma^{-1}Pic(K_{\widehat{\ell}})$ is a connective cover and hence an equivalence on $L_{K(1)}$, so $log_{K_{\widehat{\ell}}}$ naturally extends to a map $Pic(K_{\widehat{\ell}})\rightarrow \Sigma K_{\widehat{\ell}}$; and by the same token, $log_{L_{K(1)}S}$ naturally extends to a map $Pic(L_{K(1)}Sp)\rightarrow \Sigma L_{K(1)}S$, which we take as our $log$.

Now, fix a generator $u\in\mathbb{Z}_\ell^\times$.  Then the functoriality of $log_A$ shows that the above maps are compatible with the unit $L_{K(1)}Sp\rightarrow Mod_{K_{\widehat{\ell}}}$ as well as the action of $\psi^u$ on $Mod_{K_{\widehat{\ell}}}$, and hence canonically extend to a map of null-composite sequences
$$\xymatrix{ Pic(L_{K(1)}Sp)\ar[r]\ar[d]^-{log} & Pic(K_{\widehat{\ell}})\ar[r]^-{\psi^u-1}\ar[d] & Pic(K_{\widehat{\ell}})\ar[d] \\
\Sigma L_{K(1)}S\ar[r] & \Sigma K_{\widehat{\ell}}\ar[r]^-{\psi^u-1} & \Sigma K_{\widehat{\ell}},}$$
where the bottom map is the suspension of the fiber sequence of \ref{k1sph}.  (The top map is also a fiber sequence, modulo $\pi_{-1}$-issues --- compare \cite{hms} Prop.\ 2.1 --- but we don't need this.)

Now, the horizontal maps in the left-hand square are isomorphisms on $\pi_1$, so the claim about $\pi_1(log)$ reduces to Proposition \ref{rzk}.  As for the claim about $\pi_0(log)$, note that Bott periodicity provides an equivalence $S^2\wedge K_{\widehat{\ell}}\simeq K_{\widehat{\ell}}$ of invertible $K_{\widehat{\ell}}$-modules under which the Adams operation $\psi^u$ gets multiplied by $u$; thus the image of $[S^2]\in \pi_0Pic(L_{K(1)}Sp)$ in the homotopy fiber of $Pic(K_{\widehat{\ell}})\overset{\psi^u-1}{\longrightarrow} Pic(K_{\widehat{\ell}})$ lifts along the boundary map to the class $u\in \pi_1Pic(K_{\widehat{\ell}})\simeq\mathbb{Z}_{\ell}^\times$; then by functoriality of the boundary map we are again reduced to Proposition \ref{rzk}.
\end{proof}

Though the above proposition is all we will need from $log$, it might be helpful to keep in mind that a stronger statement is true: in fact $log$ identifies $Pic(L_{K(1)}Sp)_{\widehat{\ell}}$ with the connective cover of $\Sigma L_{K(1)}S$.  This will fall out of our analysis of the image of J; or even better, \cite{log} Thm.\ 1.9 implies the more precise claim that $\pi_{2k}(log)$ induces an isomorphism $\pi_{2k}Pic(L_{K(1)}Sp)\simeq\pi_{2k-1}L_{K(1)}S$ (for $k>0$) which differs from the more prosaic one (\ref{pic}) by a factor of $1-\ell^{k-1}$.

\subsection{The image of J}

Now we turn to our primary consideration.  Recall that $\ell$ is a fixed odd prime.

\begin{definition}
Let $n$ be a positive integer.  We define a bunch of subgroups of $\pi_nS_{\widehat{\ell}}\simeq\pi_{n+1}Pic(Sp_{\widehat{\ell}})$ (\ref{pic}) as follows:
\begin{enumerate}
\item First, we let $Im(J)_n\subseteq \pi_nS_{\widehat{\ell}}$ denote the image on $\pi_{n+1}$ of any map $\Sigma S_{\widehat{\ell}}\rightarrow Pic(Sp_{\widehat{\ell}})_{\widehat{\ell}}$ classifying a generator of the free $\mathbb{Z}_\ell$-module of rank one $\pi_1Pic(Sp_{\widehat{\ell}})_{\widehat{\ell}}\simeq(\mathbb{Z}_\ell^\times)_{(\ell)}$;
\item And second, for $p\leq\infty$ we let $Im(J_{\mathbb{Q}_p})_n\subseteq\pi_nS_{\widehat{\ell}}$ denote the image on $\pi_{n+1}$ of the $\ell$-completion of $J_{\mathbb{Q}_p}\colon K(\mathbb{Q}_p)_{\geq 2}\rightarrow Pic(Sp_{\widehat{\ell}})$.\end{enumerate}
\end{definition}

Now let us restate and prove Theorem \ref{imj}.

\begin{theorem}\label{trueimj}
Let $n>0$.  Then:
\begin{enumerate}
\item The subgroup $Im(J)_n$ of $\pi_nS_{\widehat{\ell}}$ maps isomorphically to $\pi_nL_{K(1)}S$;
\item For each $p\leq\infty$ we have $Im(J_{\mathbb{Q}_p})_n=Im(J)_n$, except possibly when $p=\ell$ and $n$ is even (in which case $Im(J)_n=0$, but no clue about $Im(J_{\mathbb{Q}_\ell})_n$).
\end{enumerate}
\end{theorem}

We recall that the statement about $J_{\mathbb{R}}$ is essentially classical (\cite{adams}), as are the statements about  $J_{\mathbb{Q}_p}$ for $p$ a generator of $\mathbb{Z}_\ell^\times$ (\cite{snaith}, \cite{mqrt}).  However, we do provide new proofs, based on Proposition \ref{k1analysis}.  Incidentally, I thank Jacob Lurie for indicating to me that Rezk's study of the $K(1)$-local logarithm could be used to analyze the real image of J.

\begin{proof}

First, here is a sketch.  There are three cases: $p=\infty$, $p=\ell$, and $p\neq\ell,\infty$.  In all three cases we will use the map $log\colon Pic(L_{K(1)}Sp)\rightarrow\Sigma L_{K(1)}S$ to detect classes in $Im(J_{\mathbb{Q}_p})_n$, thereby bounding this group from below.  Then in the first two cases we will also find a copy of (a cover of) $\Sigma L_{K(1)}S$ on the other side of $Pic(Sp_{\widehat{\ell}})$, letting us also bound $Im(J_{\mathbb{Q}_p})_n$ from above and relate it to $Im(J)_n$; but in the last case our method of bounding from above is more indirect, using the product formula of the next section to reduce to the previous cases $p=\ell,\infty$.\\

We start with $p=\infty$.  We will use $J_{\mathbb{R}}$ to verify claim 1 and simultaneously establish claim 2 for $p=\infty$.  Note that we can replace $J_\mathbb{R}\colon K(\mathbb{R})\rightarrow Pic(Sp_{\widehat{\ell}})$ by its precomposition $J_\mathbb{C}$ with the forgetful map $K(\mathbb{C})\rightarrow K(\mathbb{R})$; indeed, in obvious notation we clearly have $Im(J_{\mathbb{C}})_n\subseteq Im(J_{\mathbb{R}})_n$, but also $Im(J_{\mathbb{R}})_n\subseteq 2Im(J_{\mathbb{C}})_n=Im(J_{\mathbb{C}})_n$, so the two images coincide.

Now, fix a generator $u\in\mathbb{Z}_\ell^\times$.  Then by Lemma \ref{discreteadams}, the map $K(\mathbb{C})_{\widehat{\ell}}\rightarrow Pic(Sp_{\widehat{\ell}})_{\widehat{\ell}}$ induced by $J_{\mathbb{C}}$ is invariant under $\psi^u$; thus, referencing the fiber sequence in \ref{k1sph}, it factors via the boundary map as
$$K(\mathbb{C})_{\widehat{\ell}}\rightarrow (\Sigma L_{K(1)}S)_{\geq 0}\rightarrow Pic(Sp_{\widehat{\ell}})_{\widehat{\ell}}.$$
We complete this to the following diagram:
$$\xymatrix{ K(\mathbb{C})_{\widehat{\ell}}\ar[rd] & & & & \\
& (\Sigma L_{K(1)}S)_{\geq 0}\ar[r] & Pic(Sp_{\widehat{\ell}})_{\widehat{\ell}}\ar[r] & Pic(L_{K(1)}Sp)_{\widehat{\ell}}\ar[r]^-{log} & \Sigma L_{K(1)}S.\\
\Sigma S_{\widehat{\ell}}\ar[ur] & & & & }$$

Then Lemma \ref{lowreal} (on $\pi_0J_{\mathbb{R}}$) and  Proposition \ref{k1analysis} (on $\pi_0log$) imply that the long horizontal composition hits a generator on $\pi_0$.  But the set of homotopy classes of maps $(\Sigma L_{K(1)}S)_{\geq 0}\rightarrow \Sigma L_{K(1)}S$ is a free $\mathbb{Z}_\ell$-module on one generator given by the connective cover map; thus we deduce that the long horizontal composition is a unit multiple of the connective cover map, and hence an isomorphism on homotopy groups in nonnegative degrees.

That assertion is the key.  Indeed, fix $n>0$.  Then we find:
\begin{enumerate}
\item Since $\#\pi_{n+1} Pic(L_{K(1)}Sp)_{\widehat{\ell}}=\#\pi_{n+1}\Sigma L_{K(1)}S$, the map $\pi_{n+1}(log)$ is also an isomorphism, and so we can check the first claim $Im(J)_n\overset{\sim}{\longrightarrow} \pi_nL_{K(1)}S$ after applying $log$;
\item Similarly, we deduce that $(\Sigma L_{K(1)}S)_{\geq 0}\rightarrow Pic(Sp_{\widehat{\ell}})_{\widehat{\ell}}$ is an isomorphism on $\pi_1$, so the bottom diagonal map composed with the first horizontal map can be used to define $Im(J)_n$;
\item Also, considering the middle horizontal map we deduce that $\pi_nS_{\widehat{\ell}}\rightarrow \pi_nL_{K(1)}S$ is surjective, so the bottom diagonal map is surjective on $\pi_{n+1}$.
\end{enumerate}
Then since the top diagonal map is also surjective on $\pi_{n+1}$, we conclude again from the assertion the desired statements $Im(J_{\mathbb{R}})_n=Im(J)_n\overset{\sim}{\longrightarrow}\pi_nL_{K(1)}Sp$.\\

Now we use $J_{\mathbb{Q}_\ell}$ to reprove claim 1 and also establish claim 2 for $p=\ell$.  Note that since we are only interested in $\ell$-completions, by definition (\ref{padicj}) we can replace $J_{\mathbb{Q}_\ell}$ by $J_{\mathbb{Z}_\ell}\colon K(\mathbb{Z}_\ell)\rightarrow Pic(Sp_{\widehat{\ell}})$.  Now, recall that work of B\"{o}kstedt, Hesselholt, Hsiang, and Madsen gives an equivalence of spectra
$$K(\mathbb{Z}_\ell)_{\widehat{\ell}}\simeq (L_{K(1)}S)_{\geq 0}\times (\Sigma L_{K(1)}S)_{\geq 1}\times (\Sigma K_{\widehat{\ell}})_{\geq 3};$$
see \cite{madsen} for a survey.  Thus any map $\Sigma S_{\widehat{\ell}}\rightarrow K(\mathbb{Z}_\ell)_{\widehat{\ell}}$ classifying a generator of $(\mathbb{Z}_{\ell}^\times)_{(\ell)}$ factors uniquely up to homotopy through a map $(\Sigma L_{K(1)}S)_{\geq 1}\rightarrow K(\mathbb{Z}_\ell)_{\widehat{\ell}}$ inducing an isomorphism on $\pi_{n+1}$ for $n$ odd.  But then combining the second part of Lemma \ref{lowwild} (on $\pi_1(J_{\mathbb{Z}_\ell})$) with Proposition \ref{k1analysis} (on $\pi_1(log)$) we see that the composition of all of the maps except the first in
$$\Sigma S_{\widehat{\ell}}\rightarrow (\Sigma L_{K(1)}S)_{\geq 1}\rightarrow K(\mathbb{Z}_\ell)_{\widehat{\ell}}\overset{J_{\mathbb{Z}_\ell}}{\longrightarrow} Pic(Sp_{\widehat{\ell}})_{\widehat{\ell}}\rightarrow Pic(L_{K(1)}Sp)_{\widehat{\ell}}\overset{log}{\longrightarrow} \Sigma L_{K(1)}S$$
is surjective on $\pi_1$ and hence equal to a unit times the $0$-connected cover map; then we can reconclude the first claim and establish the second claim for $p=\ell$ just as in the case of $J_\mathbb{R}$.\\

Finally, we need to prove claim 2 for $p\neq\ell,\infty$.  This will be a bit different:  we will show first that $Im(J_{\mathbb{Q}_p})_n=0$ for $n$ even, then that $Im(J_{\mathbb{Q}_p})_n\twoheadrightarrow\pi_nL_{K(1)}S$ for all $n$, and finally that $Im(J_{\mathbb{Q}_p})_n\subseteq Im(J_{\mathbb{Q}_\ell})_n+Im(J_{\mathbb{R}})_n$ for all $n$; this suffices, given the previous work.

Note that since the boundary map $K(\mathbb{Q}_p)_{\widehat{\ell}}\rightarrow \Sigma K(\mathbb{F}_p)_{\widehat{\ell}}$ is surjective on homotopy groups (see e.g.\ \cite{soule} Prop.\ 4), we can replace $J_{\mathbb{Q}_p}$ with $J_{\mathbb{F}_p}\colon K(\mathbb{F}_p)\rightarrow \Sigma^{-1}Pic(Sp_{\widehat{\ell}})$.  Now, recall that Quillen in \cite{quillenfinite} has (essentially) produced an equivalence of spectra
$$K(\mathbb{F}_p)_{\widehat{\ell}}\simeq (A_p)_{\geq 0},$$
where for $u\in\mathbb{Z}_\ell^\times$ we let $A_u$ denote the fiber of $K_{\widehat{\ell}}\overset{\psi^u-1}{\longrightarrow}K_{\widehat{\ell}}.$  In particular $K(\mathbb{F}_p)_{\widehat{\ell}}$ has no homotopy groups in positive even degrees, which verifies that $Im(J_{\mathbb{Q}_p})_n=0$ for $n$ even.  But furthermore, if $d$ denotes the index of the subgroup topologically generated by $p$ in $\mathbb{Z}_\ell^\times$, then Lemma \ref{lowtame} (on $\pi_0J_{\mathbb{F}_p}$) and Proposition \ref{k1analysis} (on $\pi_1log$) imply that, on $\pi_0$, the composition
$$K(\mathbb{F}_p)_{\widehat{\ell}}\rightarrow \Sigma^{-1}Pic(L_{K(1)}Sp)_{\widehat{\ell}}\overset{log}{\longrightarrow} L_{K(1)}S$$
generates the subgroup $d\mathbb{Z}_\ell$ of $\pi_0 L_{K(1)}S\simeq \mathbb{Z}_\ell$; however, the set of homotopy classes of maps $A_p\rightarrow L_{K(1)}S$ is a free $\mathbb{Z}_\ell$-module on the norm map (defined by summing over the action of the $\mathbb{Z}_\ell^\times/\langle p\rangle$-Adams operations), which has exactly this effect on $\pi_0$ and is surjective on higher homotopy groups.  (Indeed, if $u\in\mathbb{Z}_\ell^\times$ is a generator then we can replace $p$ by $u^d$, as these generate the same subgroup; then this amounts to the identity $(u^d)^m-1=(u^m-1)(1+u^m+\ldots+(u^m)^{d-1})$.)  Thus we deduce that the displayed map is surjective on higher homotopy groups, and hence that $Im(J_{\mathbb{Q}_p})_n$ surjects onto $\pi_nL_{K(1)}S$, since by the above discussion (either of $J_\mathbb{R}$ or of $J_{\mathbb{Q}_\ell}$ or indeed of Rezk's formula), $\pi_{n+1}log$ is an isomorphism.

Thus, to finish it suffices to show that $Im(J_{\mathbb{Q}_p})_n\subseteq Im(J_{\mathbb{Q}_\ell})_n+Im(J_{\mathbb{R}})_n$.  For this, we recall that Soul\'{e} has shown (\cite{soule} Thm.\ 3) that the boundary map $K_{n+1}(\mathbb{Q})\rightarrow\oplus_{p<\infty}K_n(\mathbb{F}_p)$ associated the the localization sequence for $\mathbb{Z}\subseteq\mathbb{Q}$ is surjective; thus, given a class $x\in \pi_nK(\mathbb{F}_p)_{\widehat{\ell}}=K_n(\mathbb{F}_p)_{(\ell)}$ we can lift it along the boundary map to a class $\overline{x}\in K_{n+1}(\mathbb{Q})_{(\ell)}$ whose image along the boundary maps for any other prime are trivial.  But then we have $J_{\mathbb{Q}_{p'}}(\overline{x})=0$ for any prime $p'\neq p,\ell,\infty$, so the product formula (Theorem \ref{product}), proved in the next section, gives $J_{\mathbb{F}_p}(x)=J_{\mathbb{Q}_p}(\overline{x})=-J_{\mathbb{R}}(\overline{x})-J_{\mathbb{Q}_\ell}(\overline{x})$, implying the desired.
\end{proof}

The last step of the above proof, though it has the advantage of demonstrating the content of the product formula, is intrinsically unsatisfying in that it uses the other J's to prove a fact about $J_{\mathbb{Q}_p}$.  This would not be necessary if we could establish that $J_{\mathbb{F}_p}\colon K(\mathbb{F}_p)_{\widehat{\ell}}\rightarrow \Sigma^{-1}Pic(Sp_{\widehat{\ell}})_{\widehat{\ell}}$ factors through $(L_{K(1)}S)_{\geq 0}$ via the norm map, since then we could argue as in the other cases.  Of course, this factoring is trivial when $p$ generates $\mathbb{Z}_\ell^\times$, but otherwise an argument is needed.  It would suffice to establish the following conjecture, which permits reduction to the case of the $\ell$-complete $J_{\mathbb{C}}$, identified with the $\ell$-adic \'{e}tale $J_{\overline{\mathbb{F}_p}}$:

\begin{conjecture} Let $p$ be a prime, and $\ell$ a prime different from $p$.  Then $J_{\mathbb{F}_p}:K(\mathbb{F}_p)\rightarrow (S_{\widehat{\ell}})^\times$ identifies with the group completion of the map $Vect_{\mathbb{F}_p}^\sim\rightarrow Aut(S_{\widehat{\ell}})$ which sends $V$ to the effect of the (geometric) Frobenius map on the stable $\ell$-adic \'{e}tale homotopy type of the cofiber of $(V_{\overline{\mathbb{F}_p}}- 0)\rightarrow V_{\overline{\mathbb{F}_p}}$.\end{conjecture}

Thinking of this cofiber as representing the compactly supported homotopy type of the $\overline{\mathbb{F}_p}$-variety $V_{\overline{\mathbb{F}_p}}$, whose Frobenius fixed points are $V$ as a set, we see that this conjecture amounts to something of a homotopical amplification of the simplest special case of the Lefschetz fixed point formula in \'{e}tale cohomology.  It may also be interesting to note that this conjecture, if true, provides an unstable incarnation of $J_{\mathbb{F}_p}$, something which does not exist a priori.

Theorem \ref{trueimj} also leaves open the question of just what the $\ell$-completion of $J_{\mathbb{Q}_\ell}$ does on homotopy groups in odd degrees.  In terms of the above-referenced calculation of $K(\mathbb{Z}_\ell)_{\widehat{\ell}}$, what's left is understanding what $J_{\mathbb{Z}_{\ell}}$ does to the pieces $(L_{K(1)}S)_{\geq 0}$ and $(\Sigma K_{\widehat{\ell}})_{\geq 3}$.  The first piece comes from the unit of $K(\mathbb{Z}_\ell)$, and is fairly easy to analyze: on it, $J_{\mathbb{Z}_\ell}$ is trivial on homotopy groups (e.g.\ because the product formula shows that it factors through $K(\mathbb{R})_{\widehat{\ell}}$), but nonetheless nontrivial as a map of spectra, since after composing with $log$ it classifies a generator of $\pi_{-1}L_{K(1)}S$.  As for the last piece, work of Soul\'{e} (\cite{soule2}) shows that it corresponds in a precise sense to the system of norm-compatible integral units in the $\ell$-cyclotomic tower of $\mathbb{Q}_\ell$, but I don't know how to make the required analysis of its behavior under $J_{\mathbb{Z}_\ell}$.

We have also completely neglected to analyze the $J_{\mathbb{Q}_p}$ at the prime $\ell=2$.  The reason is that certain intricacies arise at $\ell=2$, and the details haven't been seen through.  The outcome, however, should be that for $p<\infty$ the $p$-adic image of $J$ agrees with the complex image of J, which is generally smaller than the real image of J.  For relevant literature, see \cite{snaith} for $p\neq 2$ and \cite{2adic} for $p=2$.

\section{Proof of the product formula}\label{proof}

In this section we prove Theorem \ref{product}, which can be restated as follows (see Appendix \ref{sum}): the map
$$K(\mathbb{Q})_{\geq 2}\rightarrow\prod_{p\leq\infty} K(\mathbb{Q}_p)_{\geq 2}\overset{J_{\mathbb{Q}_p}}{\longrightarrow}\prod_{p\leq\infty} Pic(Sp)$$
has a canonical lifting along the natural map $\vee_{p\leq\infty}Pic(Sp)\rightarrow\prod_{p\leq\infty} Pic(Sp)$, and there is a canonical trivialization of the resulting composite $K(\mathbb{Q})_{\geq 2}\rightarrow\vee_{p\leq\infty}Pic(Sp)\rightarrow Pic(Sp)$ with the ``sum'' map.

Let us start the proof with a series of reductions.  Note that we can clearly replace $Pic(Sp)$ by $Pic(Sp)_{\geq 2}$ in the above.  Then by Serre finiteness we have $Pic(Sp)_{\geq 2}\simeq\prod_{\ell<\infty}Pic(Sp_{\widehat{\ell}})_{\geq 2}$ and also $\vee_{p\leq\infty} Pic(Sp)_{\geq 2}\simeq \prod_{\ell<\infty}\vee_{p\leq\infty} Pic(Sp_{\widehat{\ell}})_{\geq 2}$; thus the problem breaks up over primes $\ell$, so we can further fix a prime $\ell$ and replace $Pic(Sp)_{\geq 2}$ by $Pic(Sp_{\widehat{\ell}})_{\geq 2}$.  Then for $p\neq \ell$ we can also replace $J_{\mathbb{Q}_p}$ with $J_{\mathbb{Q}_p}^{tame}$, since the target is now away-from-$p$.

But furthermore, in the localization sequence $K(\mathbb{F}_\ell)\rightarrow K(\mathbb{Z}_{(\ell)})\rightarrow K(\mathbb{Q})$ the fiber is prime-to-$\ell$ in positive degrees (\cite{quillenfinite}); hence, the target $Pic(Sp_{\widehat{\ell}})_{\geq 2}$ being $\ell$-local, we can replace the source $K(\mathbb{Q})_{\geq 2}$ by $K(\mathbb{Z}_{(\ell)})_{\geq 2}$ in our desired statement.  But then only $J_{\mathbb{Z}_{\ell}}$ is relevant, not the full $J_{\mathbb{Q}_{\ell}}$; and in fact, since $Pic(Sp_{(\ell)})_{\geq 2}\simeq Pic(Sp_{\widehat{\ell}})_{\geq 2}$, we can replace $J_{\mathbb{Z}_\ell}$ by its $\ell$-local analog $K(\mathbb{Z}_{(\ell)})\rightarrow Pic(Sp_{(\ell)})$ --- given by sending a finite free $\mathbb{Z}_{(\ell)}$-module $M$ to the inverse of the $\ell$-local stable top cell of $BM$ --- and then also ignore the passage to the $1$-connected cover if we want.

All told, Theorem \ref{product} reduces to the case $P=\{\text{all primes but }\ell\}$ of following statement:

\begin{theorem}\label{tori}
Let $P$ be a set of primes.  Then the map $TC\colon K(\mathbb{Z}[1/P])\rightarrow Pic(Sp[1/P])$ which sends a finite free $\mathbb{Z}[1/P]$-module $M$ to the $P$-inverted stable top cell of $BM$ (c.f.\ \ref{wild}) is canonically homotopic to the product of $J_{\mathbb{R}}$ with the infinite product over $p\in P$ of $J_{\mathbb{Q}_p}^{tame}$, these J's being implicitly restricted to $K(\mathbb{Z}[1/P])$ and composed to land in $Pic(Sp[1/P])$.
\end{theorem}

We can reinterpret this statement topologically: it is saying that, for tori up to $P$-isogeny, the $P$-inverted stable top cell identifies with the one-point compactification of the Lie algebra, up to an error term coming from the possible kernels of $P$-isogenies; and in fact, this is how the proof will go.  Here is a picture:

\begin{figure}[H]
\centering
\includegraphics[scale=.5]{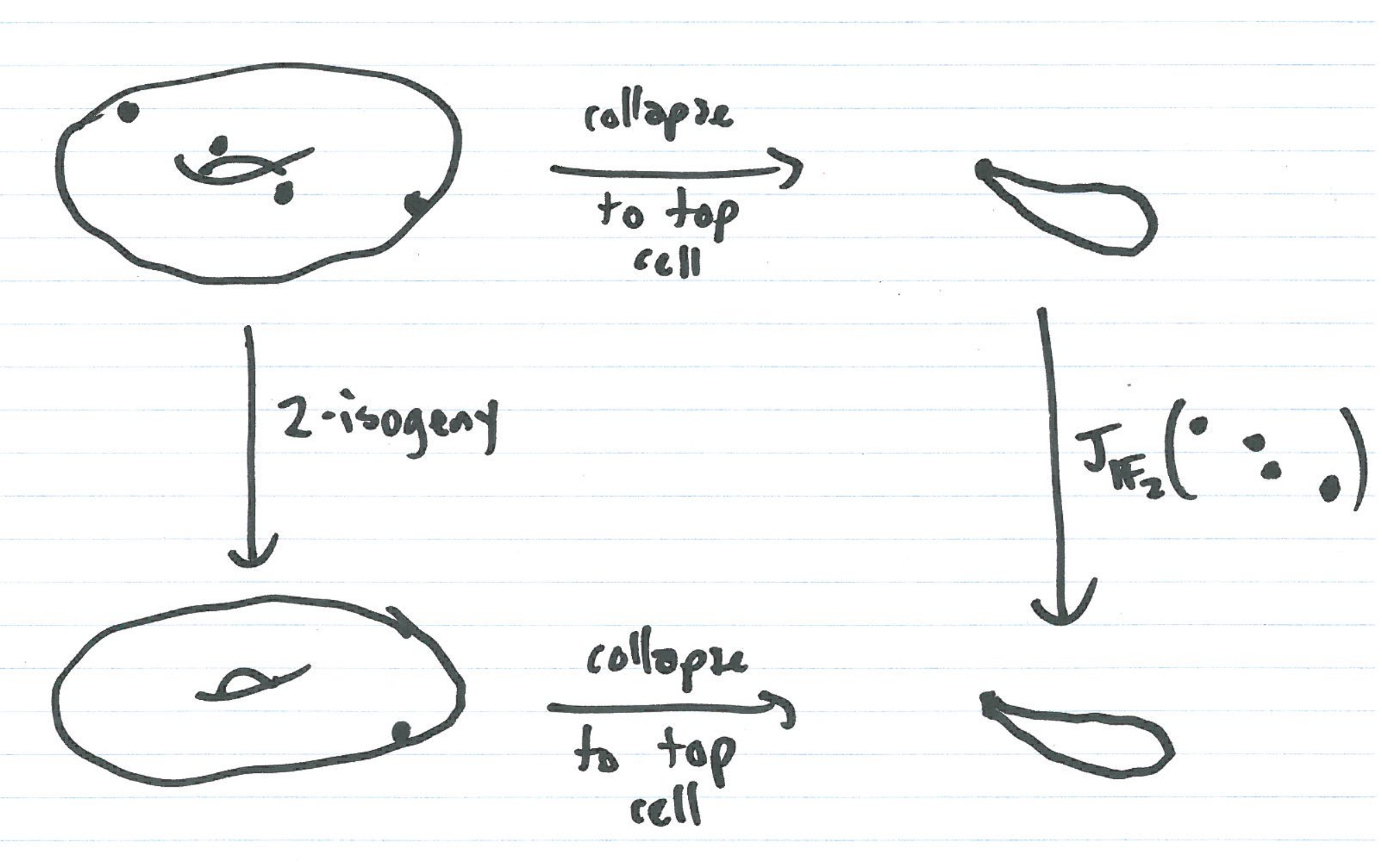}
\caption{An illustration of the product formula:  this diagram stably commutes after $2$-inversion.  Here the pictured isogeny is multiplication by two on some fixed $2$-torus.}
\end{figure}

\begin{proof}
By functoriality of the localization sequence in algebraic K-theory (\cite{Q} Section 5), for $p'\in P$ the map $J_{\mathbb{Q}_{p'}}^{tame}$ restricted to $K(\mathbb{Z}[1/P])$ identifies with the composition
$$K(\mathbb{Z}[1/P])\overset{\partial}{\longrightarrow} \vee_{p\in P}\Sigma K(\mathbb{F}_{p})\rightarrow \Sigma K(\mathbb{F}_{p'})\overset{J_{\mathbb{F}_{p'}}}{\longrightarrow} Pic(Sp[1/p']),$$
where $\partial$ is the boundary map of the localization sequence $\vee_{p\in P}K(\mathbb{F}_p)\rightarrow K(\mathbb{Z})\rightarrow K(\mathbb{Z}[1/P])$; thus the infinite product of the $J_{\mathbb{Q}_p}^{tame}$ in the statement of the theorem is canonically defined by $\partial$, and it suffices to produce a canonical homotopy between $TC$ and the product of $J_\mathbb{R}$ with the composition $K(\mathbb{Z}[1/P])\overset{\partial}{\longrightarrow}\vee_{p\in P} \Sigma K(\mathbb{F}_p)\overset{J_{\mathbb{F}_p}}{\longrightarrow} Pic(Sp[1/P])$.  Equivalently, we have to provide the following data:
\begin{enumerate}
\item After restriction to $K(\mathbb{Z})$, a canonical identification of $TC$ with $J_{\mathbb{R}}$;
\item After further restriction to $\vee_{p\in P}K(\mathbb{F}_p)$, a canonical identification of the resulting self-equivalence $0\simeq TC\simeq J_{\mathbb{R}}\simeq 0$ of the zero map $\vee_{p\in P}K(\mathbb{F}_p)\rightarrow Pic(Sp[1/P])$ with the map $\vee_{p\in P}K(\mathbb{F}_p)\overset{J_{\mathbb{F}_p}}{\longrightarrow} Aut(S[1/P])$.
\end{enumerate}
Now let's recall exactly where this localization sequence $\vee_{p\in P}K(\mathbb{F}_p)\rightarrow K(\mathbb{Z})\rightarrow K(\mathbb{Z}[1/P])$ comes from.  Let $Ab_P^f$ denote the exact category of finite $P$-primary abelian groups, $Ab_P$ the exact category of finitely generated abelian groups with $P$-primary torsion, and $Vect_{\mathbb{Z}[1/P]}$ the exact category of finitely generated free $\mathbb{Z}[1/P]$-modules.  Then there are exact functors
$$Ab^f_P\rightarrow Ab_P\rightarrow Vect_{\mathbb{Z}[1/P]}$$
with canonically trivial composition; but combining the equivalence (\cite{gray} pp.\ 8) between the group completion and Q-construction definitions of algebraic K-theory with the devissage theorem (\cite{Q} Thm.\ 3) we see that the natural maps $\vee_{p\in P}K(\mathbb{F}_p)\rightarrow K(Ab_P^f)$ and $K(\mathbb{Z})\rightarrow K(Ab_P)\rightarrow K(Mod_\mathbb{Z}^{fg})$ and $K(\mathbb{Z}[1/P])\rightarrow K(Vect_{\mathbb{Z}[1/P]})\rightarrow K(Mod_{\mathbb{Z}[1/P]}^{fg})$ are all equivalences; thus the localization theorem (\cite{Q} Thm.\ 5) implies that the induced sequence on Waldhausen K-theory (our chosen model -- see Appendix \ref{sdot})
$$K(Ab_P^f)\rightarrow K(Ab_P)\rightarrow K(Vect_{\mathbb{Z}[1/P]})$$
is a fiber sequence and indeed was the localization sequence we meant all along via the above equivalences.

Now, consider the functor $\mathbb{B}\colon Ab_P\rightarrow Orb_*$ from $Ab_P$ to pointed orbifolds which sends an $M\in Ab_P$ to the orbifold quotient of $M\otimes\mathbb{R}$ by its translation action of $M$, pointed by the image of the origin $0\in M\otimes\mathbb{R}$.  (Thus, if $M$ is a finitely generated free $\mathbb{Z}$-module then $\mathbb{B}M$ is a torus, whereas if $M$ is finite then $\mathbb{B}M=BM$ is just the usual classifying stack --- but in any case $\mathbb{B}M$ has the homotopy type of $BM$.)  Now, let us work with the Bousfield localization $L_{H\mathbb{Z}[1/P]}$ instead of $P$-inversion; this is permissible since they give the same $Pic$ (see \ref{bousfact} and \ref{pic}).  Then, referencing the terminology of Theorem \ref{unithm}, we claim that $\mathbb{B}$ actually gives an exact functor $Ab_P\rightarrow CUO_*$.  Indeed:
\begin{enumerate}
\item Each $\mathbb{B}M$ is $L_{H\mathbb{Z}[1/P]}$-compact, being the product of a torus (just plain compact) with $BM$ for a finite group $M$ whose order is invertible in $\mathbb{Z}[1/P]$ (see the discussion after Definition \ref{lecpct});
\item Each $\mathbb{B}M$ has $L_{H\mathbb{Z}[1/P]}$-unipotent duality, since tori are parallelizable (see the discussion after Definition \ref{unidual});
\item $\mathbb{B}$ preserves pullbacks by epimorphisms $M\twoheadrightarrow M'$ and sends $0$ to the point:  indeed, the first claim is straightforward to check after pulling back to $M'\otimes\mathbb{R}$, which covers $\mathbb{B}M'$; and the second claim is trivial.
\item For every epimorphism $M\twoheadrightarrow M'$ the induced map $\mathbb{B}M\rightarrow\mathbb{B}M'$ is an $L_{H\mathbb{Z}[1/P]}$-unipotent submersion:  indeed, ``submersion'' is clear by pulling back to $M'\otimes\mathbb{R}$, and for ``$L_{H\mathbb{Z}[1/P]}$-unipotent'' we can use the criterion of Proposition \ref{criterion}, which applies because, setting $M''=ker(M\twoheadrightarrow M')$, the monodromy action on  the fiber $\mathbb{B}M''$ is given by translation and hence is element-wise trivial up to homotopy, $\mathbb{B}M''$ being connected.
\end{enumerate}

Thus, by Theorem \ref{unithm}, the following two maps $K(Ab_P)\rightarrow Pic(L_{H\mathbb{Z}[1/P]}Sp)$ are canonically homotopic:
\begin{enumerate}
\item First, $M\mapsto Sph^{top}(\mathbb{B}M)$ (see Theorem \ref{stacksphere});
\item Second, $M\mapsto Sph^{alg}(C^*(\mathbb{B}M))$ (see Theorem \ref{algsphere}).
\end{enumerate}
However, the first map identifies with $M\mapsto J_{\mathbb{R}}(M\otimes\mathbb{R})$:  indeed, the additivity theorem (\ref{add}) applied first to the fiber sequence $M\rightarrow M\otimes\mathbb{R}\rightarrow \mathbb{B}M$ and then to the fiber sequence $*\rightarrow *\rightarrow M$ identifies $[\mathbb{B}M]$ with $[M\otimes\mathbb{R}]$ in $K(Orb_*)$; but $Sph^{top}(M\otimes\mathbb{R})$ identifies with $J_\mathbb{R}(M\otimes\mathbb{R})$ by construction.  And the second map identifies with $M\mapsto TC(M[1/P])$:  indeed, the map $BM\rightarrow B(M[1/P])$ is a $\mathbb{Z}[1/P]$-homology isomorphism, so $C^*(BM[1/P])\rightarrow C^*(BM)$ is an equivalence.

Thus Theorem \ref{unithm} has provided the desired homotopy $TC\simeq J_{\mathbb{R}}$ on restriction to $K(Ab_P)$.  Then all that remains is to identify, for each $p\in P$, the resulting self-equivalence
$$0\simeq TC\simeq J_{\mathbb{R}}\simeq 0$$
of the zero map $K(\mathbb{F}_p)\rightarrow Pic(L_{H\mathbb{Z}[1/P]}Sp)$ with $J_{\mathbb{F}_p}\colon K(\mathbb{F}_p)\rightarrow Aut(S[1/P])$.  However, Proposition \ref{bgshf} calculates the last equivalence, and Proposition \ref{algtriv} calculates the composition of the first two; putting it together, that self-equivalence of $0$ identifies with the group completion of the map $Vect_{\mathbb{F}_p}^\sim\rightarrow Aut(S[1/P])$ given by sending $M\in Vect_{\mathbb{F}_p}^\sim$ (viewed merely as a finite group) to the $P$-inversion of the composition
$$S\simeq (\prod_M S)^{M}\rightarrow \prod_M S\rightarrow (\prod_MS)_{M}\simeq (\vee_M S)_{M}\simeq S,$$
and this does indeed agree with the description of $J_{\mathbb{F}_p}$ (\ref{tame}).
\end{proof}

\section{Closing remarks and speculations}\label{remarks}

\subsection{Extracting tangible number-theoretic consequences from the product formula (other than quadratic reciprocity)}

I don't know how it can be done.  The natural map from Milnor K-theory to Quillen K-theory of local fields with finite coefficients is zero in degrees $\geq 3$ (e.g.\ by Bloch-Kato), and this rules out naive extraction of ``higher Hilbert symbols'' from the $\pi_n J_{\mathbb{Q}_p}$.

\subsection{The function field case}

Let us sketch the function field result --- due to Gillet (\cite{gillet} Section 2.3) --- that inspired Theorem \ref{product}.  Fix a field $k$.  Then for any complete discrete valuation field $L$ whose residue field $k_L$ is a finite extension of $k$ we can define a map $J_L\colon K(L)\rightarrow \Sigma K(k)$ as the composition
$$K(L)\rightarrow \Sigma K(k_L)\rightarrow \Sigma K(k),$$
where the first map is the boundary map in the localization sequence for the ring of integers of $L$, and the second map is the forgetful map.

The product formula in the function field case then says that if $X$ is any proper smooth curve over $k$, with fraction field say $K$, then the infinite sum over all closed points $x\in X$ of the compositions
$$K(K)\rightarrow K(K_x)\overset{J_{K_x}}{\longrightarrow} \Sigma K(k)$$
is canonically trivial.  (This follows from the formalism of proper pushforwards in algebraic K-theory, which implies that map $\vee_{x\in X}K(k_x)\rightarrow K(k)$ factors as
$$\vee_{x\in X}K(k_x)\rightarrow K(X)\rightarrow K(k),$$
where the first map fits into the localization sequence $\vee_{x\in X}K(\kappa_x)\rightarrow K(X)\rightarrow K(K)$.)

Thinking of the function field $K$ as like $\mathbb{Q}$ and the Laurent fields $K_x$ as like the $\mathbb{Q}_p$, the analogy with Theorem \ref{product} is clear; however, it is inexact, for two reasons:
\begin{enumerate}
\item The above maps $J_L$ are defined on the full K-theory spectra, whereas to define the $J_{\mathbb{Q}_p}$ we were obliged to pass to $1$-connected covers in concession to wild phenomena (\ref{wild});
\item The target spectra are of a different nature:  above we have $\Sigma K(k)$ whereas in Theorem \ref{product} we had $Pic(Sp)$.
\end{enumerate}
However, these differences are overcome by $K(1)$-localization (which, by Thomason's descent theorem \cite{thomason}, amounts to considering $\ell$-adic \'{e}tale K-theory instead of algebraic K-theory).  Indeed, $K(1)$-localization is invariant under passage to connective covers, obviating the first point; and as for the second point, the map $log$ (\ref{k1analysis}) furnishes an equivalence $L_{K(1)}Pic(Sp)\simeq L_{K(1)}\Sigma S$, and $L_{K(1)}\Sigma S$ is perfectly analogous to $L_{K(1)}\Sigma K(k)$ according to the well-touted tenet that ``the algebraic K-theory of the field with one element is the sphere spectrum''.

However, even better would be to have analogous proofs of these analogous results, and as far as I can tell the above application of $K(1)$-localization does nothing for this.

However, a different application of $K(1)$-localization --- on the inside of $Pic$ rather than the outside --- may do the trick.  Indeed, if we take our target to be $Pic(L_{K(1)}Sp)$ instead of $Pic(Sp)$, then we can circumvent passage to the $1$-connected covers in the definition of $J_{\mathbb{Q}_p}$ for a different reason:  the Ravenel-Wilson calculations (\cite{raw}) imply that the $E_\infty$-algebra $C^*(BM)$ in $L_{K(1)}Sp$ is a Poincar\'{e} duality algebra for $M$ any finitely generated $\mathbb{Z}_p$-module, not necessarily free; but for $M$ finite $C^*(BM)$ has canonically trivial $K(1)$-local top cell (because of the fiber sequence $M\rightarrow * \rightarrow BM$).  This means that the map $J_{\mathbb{Z}_p}\colon K(\mathbb{Z}_p)\rightarrow Pic(L_{K(1)}Sp)$ is canonically trivial on restriction to $K(\mathbb{F}_p)$, and hence, by the localization sequence, canonically extends to the desired map $J_{\mathbb{Q}_p}^{wild}\colon K(\mathbb{Q}_p)\rightarrow Pic(L_{K(1)}Sp)$ even without passage to $1$-connected covers.

A consequence is that after passing to $Pic(L_{K(1)}Sp)$ there's no tangible obstruction to realizing the following potential alternative proof of Theorem \ref{product}:  fix a prime $\ell$ to take $K(1)$-localization at, and somehow define a sort of universal J-homomorphism $J\colon K(LCA[1/\ell])\rightarrow Pic(L_{K(1)}Sp)$ (where $LCA[1/\ell]$ denotes the exact category of uniquely $\ell$-divisible locally compact abelian groups), having the property that for each $p\leq\infty$ the composition $K(\mathbb{Q}_p)\rightarrow K(LCA[1/\ell])\overset{J}{\longrightarrow} Pic(L_{K(1)}Sp)$ agrees with $J_{\mathbb{Q}_p}$.  Then the $K(1)$-local version of Theorem \ref{product} would follow by considering, for $V\in Vect_\mathbb{Q}$, the exact sequence
$$0\rightarrow V\rightarrow V\otimes\mathbb{A}\rightarrow (V\otimes\mathbb{A})/V\rightarrow 0$$
in $LCA[1/\ell]$, where $\mathbb{A}$ denotes the ring of ad\`{e}les.  In this exact sequence the first term $V$ is discrete and hence canonically trivial in K-theory by an Eilenberg swindle with infinite direct sums, and by the same token the last term $V\otimes\mathbb{A}/V$ is compact and hence canonically trivial in K-theory by an Eilenberg swindle with infinite direct products; thus the middle term is also canonically trivial by the additivity theorem, and hence the product formula holds even inside $K(LCA[1/\ell])$, i.e.\ before composing with $J$.

The above-proposed proof of $L_{K(1)}$ of Theorem \ref{product} is analogous to Tate's proof (\cite{tate}) of the product formula for valuations of $\mathbb{Q}$ using Haar measures; indeed, these proofs are formally the same once one views the theory of Haar measures for locally compact abelian groups as providing a map of spectra $K(LCA)\rightarrow \Sigma H\mathbb{R}_{>0}$.

The analogous set-up in the function field case does indeed exist, using a map $J\colon K(LCV_k)\rightarrow \Sigma K(k)$ (in all likelihood an equivalence) which can be described as follows:  $LCV_k$ denotes the exact category of locally linearly compact $k$-vector spaces, also known as Tate $k$-vector spaces (see \cite{dr}), and $J$ sends a $V\in LCV_k$ to the contractible space $Latt(V)$ of linearly compact subspaces of $V$ under inclusion together with the $\Omega^\infty K(k)$-bundle over $Latt(V)$ which associates to every inclusion $L_0\subseteq L_1$ of linearly compact subspaces of $V$ the class of $L_1/L_0$ in $\Omega^\infty K(k)$ (compare \cite{Q} Remark 5.17).

\subsection{The general number field case and roots of unity}

For any characteristic zero local field $L$ plus finite subgroup $\mu$ of roots of unity in $L$ it should be possible to define a J-homomorphism $K(L)_{\geq 2}\rightarrow Pic(Sp_\mu)$, where $Sp_\mu$ denotes the $\infty$-category of $\mu$-equivariant spectra; and for a number field $K$ together with $\mu\subseteq K^*$ the obvious product formula should hold and should recover Hilbert reciprocity for $K$ on $\pi_2$.

\subsection{$J_{\mathbb{Z}_p}$ and Lazard duality for $p$-adic Lie groups}

Recall that for $M$ a smooth compact manifold, Atiyah duality (\cite{atiyah}) promotes Poincar\'{e} duality to the level of local systems of spectra on $M$ by taking as dualizing object the fiberwise application of the real J-homomorphism to the tangent bundle of $M$ (c.f.\  \ref{cpct}).  In a similar way, for $G$ a torsion-free compact $p$-adic Lie group, Lazard's Poincar\'{e} duality for $BG$ should promote to local coefficient systems of $p$-complete spectra on the toposistic $BG$, with dualizing object given by fiberwise application of $J_{\mathbb{Z}_p}$ to the the $\mathbb{Z}_p$-Lie algebra of $G$ with its adjoint action.  And $K(n)$-locally, this should extend to an even larger class of $p$-adic Lie groups.

\subsection{Higher chromatic classes}

Homotopy theorists are interested in finding higher chromatic analogs of the J-homomorphism, which would hopefully lift certain classes in the homotopy groups of the $E(n)$-local sphere ($n\geq 2$) to the unlocalized sphere.  I have no promising ideas on this subject; let me only share the following naive observation:  after passing from curves to surfaces in the algebraic-geometric case, the natural target becomes $\Sigma^2 K(k)$ rather than $\Sigma K(k)$; thus, bearing in mind the Ausoni-Rognes idea (\cite{ar}) that the ``fraction field of complex K-theory'' is a sort of $2$-dimensional local field with algebraic K-theory of chromatic level 2, perhaps in the $E(2)$-local case one should change targets from $Pic(Sp)$ to the further delooping $Br(Sp)$ (c.f.\ forthcoming work of Gepner), or something similar.

\appendix

\section{Infinite sums of maps of spectra}\label{sum}
Here we discuss the notion of an infinite sum of maps between spectra.  (In the introduction we used the word ``product'' instead of ``sum'', but that's just because our target spectrum $Pic(Sp)$ had spectrum structure coming from the smash product (\ref{pic}).  For a general target spectrum people usually talk about maps being summed, not producted, just like for abelian groups.)

In fact, the definition we give works not just for maps of spectra, but for maps in any pointed $\infty$-category $\mathcal{C}$ with coproducts.  First let's establish notation in this context.  We will denote coproducts in $\mathcal{C}$ using $\vee$, as is conventional in our example of interest $\mathcal{C}=Sp$.  Further, if $Y$ is an object of $\mathcal{C}$ and $I$ a set, then we write $\varphi_i\colon \vee_{i\in I}Y\rightarrow Y$ for the map given by the identity on the $i^{th}$ summand and the point everywhere else, and we write $\sigma\colon \vee_{i\in I}Y\rightarrow Y$ for the map which is the identity on every summand.

\begin{definition}  Let $\{f_i\}_{i\in I}\colon X\rightarrow Y$ be a set of parallel maps in a pointed $\infty$-category $\mathcal{C}$ with coproducts.  We define the ``space of ways of summing the $\{f_i\}_{i\in I}$'' to be the space of maps $f\colon X\rightarrow\vee_{i\in I}Y$ together with homotopies $\varphi_i\circ f\simeq f_i$ for all $i\in I$.

The operation of composing $f$ with $\sigma$ then gives a natural ``sum'' map from the space of ways of summing the $\{f_i\}_{i\in I}$ to the space of maps $Map(X,Y)$.
\end{definition}

The terminology may seem awkward.  In the familiar case where $\mathcal{C}=Ab$ is the category of abelian groups, there's really no space in sight, nor even just a set: in abelian groups, the natural map $\vee_{i\in I}Y\rightarrow\prod_{i\in I}Y$ from the coproduct to the product is a monomorphism, so the space of ways of summing the $\{f_i\}_{i\in I}$ is always either empty or contractible.  Thus being able to be summed is just a condition:  the $\{f_i\}_{i\in I}$ can be summed if and only if for any $x\in X$ we have $f_i(x)=0$ for all but finitely many $i$, in which case their uniquely determined sum $g$ is given by $g(x)=\sum_{i\in I}f_i(x)$, where in reality this is only a finite sum over the nonzero $f_i(x)$'s.

In contrast, when $\mathcal{C}=Sp$ is the $\infty$-category of spectra, the map $\vee_{i\in I} Y\rightarrow\prod_{i\in I} Y$ is generally in no sense a monomorphism, so that ``being able to be summed'' is not actually a property, but structure.  To illustrate this, let's give an example of maps $\{f_i\}_{i\in I}$ that can be summed in two different ways, giving two inequivalent sums.  In fact, we will give an example where the zero map $X\rightarrow Y$ can be summed with itself infinitely often to give a nonzero map (it can, of course, also be summed with itself any number of times to give zero, by choosing the map $X\rightarrow\vee_{i\in I} Y$ to be zero).

For this, let $Y$ be any spectrum with $\pi_0(Y)\simeq\mathbb{Z}$, and let $X$ be the fiber of $\vee_{n\in\mathbb{N}}Y\rightarrow \prod_{n\in\mathbb{N}}Y$.  Then the natural map $X\rightarrow\vee_{n\in\mathbb{N}}Y$ has null composite with $\vee_{n\in\mathbb{N}}Y\rightarrow\prod_{n\in\mathbb{N}}Y$ by construction, and hence gives a way of summing an $\mathbb{N}$'s worth of zero maps from $X$ to $Y$; but on the other hand we claim that the corresponding sum $X\rightarrow\vee_{n\in\mathbb{N}}Y\overset{\sigma}{\longrightarrow} Y$ is not nullhomotopic.

Indeed, if it were, then since in spectra fiber sequences are cofiber sequences, the sum map $\vee_{n\in\mathbb{N}}Y\rightarrow Y$ would factor through $\prod_{n\in\mathbb{N}}Y$ up to homotopy.  But this is contradicted on $\pi_0$:  in abelian groups, there is no factoring of the sum map $\oplus_{n\in\mathbb{N}}\mathbb{Z}\rightarrow\mathbb{Z}$ through $\prod_{n\in\mathbb{N}}\mathbb{Z}$.  For instance, by $3$-adic considerations the element $(1,3,3^2,\ldots)$ would have to go to $-1/2$.

So infinite sums are sometimes non-unique.  It could even be that this phenomenon occurs in the case under consideration in our main theorem (Theorem \ref{product}), where we infinitely sum the different J-homomorphisms.  What our proof produces is a canonical way of summing the J-homomorphisms, together with a canonical nullhomotopy of the corresponding sum; but perhaps there's another way to sum them and get a nonzero sum.

However, this issue is a ``phantom'' one, in the technical sense that it does not arise on homotopy classes of maps from compact spectra:

\begin{proposition}\label{cpctsum} Let $\mathcal{C}$ be a stable $\infty$-category with coproducts and $\{f_i\}_{i\in I}\colon X\rightarrow Y$ a set of parallel maps in $\mathcal{C}$.  If the $\{f_i\}_{i\in I}$ can be summed to $g\colon X\rightarrow Y$, then for any compact object $K$ of $\mathcal{C}$, the maps $[K,X]\rightarrow [K,Y]$ induced by the $\{f_i\}_{i\in I}$ sum to the map $[K,X]\rightarrow [K,Y]$ induced by $g$ (this summing taking place in the category of abelian groups).
\end{proposition}
\begin{proof} For compact $K$ in $\mathcal{C}$ stable, the functor $[K,-]\colon \mathcal{C}\rightarrow Ab$ preserves coproducts.\end{proof}

\section{Waldhausen K-theory}\label{sdot}

We make the following definition, functionally equivalent to that of a category with cofibrations and weak equivalences from \cite{wald}.  We adopt the dual perspective of fibrations because of the nature of our applications (Appendix \ref{spheres}).

\begin{definition}\label{fib}
An $\infty$-category with fibrations $\mathcal{C}$ is a small pointed $\infty$-category $\mathcal{C}$ together with a collection of maps in $\mathcal{C}$ (called the fibrations), satisfying the following conditions:
\begin{enumerate}
\item A composition of two fibrations is a fibration, and every equivalence is a fibration;
\item The map $X\rightarrow *$ is a fibration for all $X\in\mathcal{C}$;
\item Pullbacks of fibrations by arbitrary maps exist in $\mathcal{C}$ and are fibrations.
\end{enumerate}
\end{definition}

Recall that $Ar[n]:=Fun(\Delta^1,[n])$ denotes the arrow category of the poset $[n]$, for $n\geq 0$; thus $Ar[n]$ identifies with the poset with objects the $(i,j)$ for $n\geq i\geq j\geq 0$ and where $(i,j)\geq (i',j')$ if and only if $i\geq i'$ and $j\geq j'$.  Now suppose that $\mathcal{C}$ is an $\infty$-category with fibrations.  Then we can make a simplicial $E_\infty$-space $S_\bullet\mathcal{C}$ as follows.  For $n\geq 0$, consider the full subcategory of the functor $\infty$-category
$$Fun(Ar[n],\mathcal{C})$$
consisting of those functors $X\colon Ar[n]\rightarrow\mathcal{C}$, denoted $(i\geq j)\mapsto X_{i,j}$, with the following properties:
\begin{enumerate}
\item For all $i$, the object $X_{i,i}$ is a point;
\item For all $i> j$ with $i\neq n$, the square $\xymatrix{ X_{i,j}\ar[r]\ar[d] & X_{i+1,j}\ar[d]\\
X_{i,j+1}\ar[r] & X_{i+1,j+1} }$ is a pullback with vertical maps fibrations.
\end{enumerate}
Cartesian product provides this full subcategory with a symmetric monoidal structure, and we let $S_n\mathcal{C}$ denote its $E_\infty$-space of objects (\ref{infcat}).

\begin{definition}\label{g}
If $\mathcal{C}$ is an $\infty$-category with fibrations, the Waldhausen K-theory of $\mathcal{C}$, denoted $K(\mathcal{C})$, is defined to be the connective spectrum corresponding (\ref{gp}) to the group-like $E_\infty$-space $\Omega|S_\bullet\mathcal{C}|$, loops on the geometric realization of $S_\bullet\mathcal{C}$.
\end{definition}

To see what this definition is doing, note first that $S_0\mathcal{C}$ is just a point, and that $S_1\mathcal{C}$ is the space of objects $\mathcal{C}^\sim$; thus to each $X\in\mathcal{C}^\sim$ is associated a loop $[X]\in \Omega|S_\bullet\mathcal{C}|\simeq\Omega^\infty K(\mathcal{C})$, the ``class of $X$ in K-theory''.  Moreover, $S_2\mathcal{C}$ serves to identify $[E]\simeq [B] + [F]$ whenever we have a fibration $E\rightarrow B$ with fiber $F$, and the higher $S_n\mathcal{C}$ encode coherences for these identifications.

By group completing $X\mapsto [X]$ with respect to cartesian product we get a canonical map of spectra
$$(\mathcal{C}^\sim)^{gp}\rightarrow K(\mathcal{C}).$$
It is generally not an equivalence:  $K(\mathcal{C})$ cares about all fiber sequences, not just those with an exhibited trivialization (though ``exhibited'' is sometimes not so important; see e.g.\ \cite{wald} Thm.\ 1.8.1 and \cite{gray} pp.\ 11).

\begin{definition}\label{exact}
A functor between two $\infty$-categories with fibrations is called exact if it preserves the point, the fibrations, and the pullbacks by fibrations.
\end{definition}

Thus an exact functor $F\colon \mathcal{C}\rightarrow\mathcal{D}$ induces a natural map $[F]\colon K(\mathcal{C})\rightarrow K(\mathcal{D})$.  The first fundamental result about Waldhausen K-theory is the following ``additivity theorem'' (\cite{wald} Prop.\ 1.3.2) concerning fiber sequences of exact functors. 

\begin{theorem}\label{add}
Let $\mathcal{C}$ and $\mathcal{D}$ be $\infty$-categories with fibrations, and let $F\rightarrow E\rightarrow B$ be an objectwise fiber sequence of exact functors from $\mathcal{C}$ to $\mathcal{D}$ with $E\rightarrow B$ an objectwise fibration.  Suppose that for all fibrations $X\rightarrow Y$ in $\mathcal{C}$, the map $E(X)\rightarrow B(X)\times_{B(Y)} E(Y)$ is a fibration in $\mathcal{D}$.

Then there is a canonical homotopy $[E]\simeq [F]+[B]$ of maps $K(\mathcal{C})\rightarrow K(\mathcal{D})$; and when restricted to $(\mathcal{C}^\sim)^{gp}$, this homotopy agrees with the one provided by the $E_\infty$-map $(F\rightarrow E\rightarrow B)\colon \mathcal{C}^\sim\rightarrow S_2\mathcal{D}$.
\end{theorem}

A proof in this $\infty$-categorical context should be forthcoming in work of Barwick and Rognes.  However, the argument of \cite{wald} Prop.\ 1.3.2 is also essentially sufficient.

One source of $\infty$-categories with fibrations is the exact categories of \cite{Q}, the fibrations being the admissible epimorphisms.  In this case, the appendix to \cite{wald} gives the equivalence between the $S_\bullet$-construction and the Q-construction of \cite{Q}; thus we can cite theorems from \cite{Q} (and its follow-up article \cite{gray}) when dealing with the Waldhausen K-theory of an exact category.

\section{A mechanism for producing stable spheres}\label{spheres}
Actually, we give two such mechanisms.  But they are sometimes equivalent, and even when not, they are produced by the same general categorical formalism.

The first mechanism builds spheres from manifolds (or orbifolds, as it turns out we'll need).  The idea is simple:  if $(M,*)$ is a pointed manifold, then the (homotopy) quotient $|M|/|M-*|$ is a sphere.  (Here we use the notation $|X|$ to indicate that we are viewing the space $X$ just as a homotopy type.)  The difficultly, however, is that we also need to encode certain compatibilities this construction $M\mapsto |M|/|M-*|$ satisfies --- mainly its multiplicativity in fiber sequences of pointed manifolds $F\rightarrow E\rightarrow B$ when $E\rightarrow B$ is a submersion --- and this is finicky.

Here is the precise statement (modulo explication of our orbifold definitions, delegated to Section \ref{orbdef}):

\begin{theorem}\label{stacksphere}
Let $Orb_*$ denote the $\infty$-category (just a $2$-category, really) of pointed orbifolds.  Then $Orb_*$ becomes an $\infty$-category with fibrations (Appendix \ref{sdot}) by taking the fibrations to be the submersions, and there is a canonical map of spectra
$$Sph^{top}\colon K(Orb_*)\rightarrow Pic(Sp).$$
When restricted to the $E_\infty$-space $Man_*^\sim$ of pointed manifolds under cartesian product, this $Sph^{top}$ identifies with the $E_\infty$-map $Man_*^\sim\rightarrow Inv(Sp)$ given by
$$(M,*)\mapsto \Sigma^\infty |M|/|M-*|.$$
\end{theorem}

This mechanism for extracting spheres from pointed manifolds is local near the point: if $U$ is an open subset of $M$ containing $*$, then there is a canonical equivalence
$$Sph^{top}(U)\simeq Sph^{top}(M).$$
(This is just excision, but it's also encoded by the above theorem, since $U\rightarrow M$ is a submersion with trivial fiber.)  In particular, a priori $Sph^{top}$ is sensitive to $M$'s topology, not just its homotopy type.

However, if we assume that $M$ is compact, then $Sph^{top}(M)$ has an alternate global description which depends only on the (pointed) homotopy type $|M|$ of $M$:  it identifies with the fiber of the inverse to the Spivak normal fibration of $|M|$ at $*$.  What's more, if $M$ is furthermore parallelizable\footnote{Even just orientable suffices, at least provided one works $H\mathbb{Z}$-locally (Proposition \ref{criterion}); and working $H\mathbb{Z}$-locally is fairly innocuous here, since spheres, suspension spectra, and spectra of maps to $S$ are all $H\mathbb{Z}$-local.}, then $Sph^{top}(M)$ can be extracted just from the (augmented) $E_\infty$-algebra in spectra
$$C^*(M):=\underline{Map}(\Sigma^\infty_+|M|,S),$$
the spectrum of maps from $|M|$ to the sphere spectrum.  Indeed, in this case, the spectrum $C_*(M):=\Sigma^\infty_+|M|$ is an invertible module over $C^*(M)$ (this is an expression of Atiyah duality), and a spectrum equivalent to $Sph^{top}(M)$ can be gotten by tensoring $C_*(M)$ up along the augmentation of $C^*(M)$ provided by the point of $M$.

This motivates our second mechanism, which will produce stable spheres from certain $E_\infty$-algebras, namely the following:

\begin{definition}\label{pdalg}
Let $\mathcal{C}$ be a presentable symmetric monoidal $\infty$-category whose tensor product preserves colimits in each variable separately.  An $E_\infty$-algebra $A$ in $\mathcal{C}$ is called a Poincar\'{e} duality algebra if the following conditions are satisfied:
\begin{enumerate}
\item The underlying object of $A$ is dualizable in $\mathcal{C}$;
\item The resulting dual $A^\vee$ is invertible as an $A$-module.
\end{enumerate}
Furthermore, a map $A\rightarrow B$ of $E_\infty$-algebras in $\mathcal{C}$ is called a Poincar\'{e} duality map if it makes $B$ into a Poincar\'{e} duality algebra in $Mod_A$.
\end{definition}

Here is the second mechanism, in the general Bousfield-local case (\ref{loc}):

\begin{theorem}\label{algsphere}
Let $L$ be a Bousfield localization of spectra, and let $PD_*$ denote the $\infty$-category of augmented Poincar\'{e} duality $E_\infty$-algebras in $LSp$.  Then $PD_*^{op}$ becomes an $\infty$-category with fibrations by taking the fibrations to be opposite to the Poincar\'{e} duality maps, and there is a canonical map of spectra
$$Sph^{alg}\colon K(PD_*^{op})\rightarrow Pic(LSp).$$
When restricted to the $E_\infty$-space $PD_*^\sim$ of augmented Poincar\'{e} duality $E_\infty$-algebras in $LSp$ under coproduct (given by tensor product), this $Sph^{alg}$ agrees with the $E_\infty$-map $PD_*^\sim\rightarrow Inv(LSp)$ given by
$$(A\rightarrow\mathbf{1})\mapsto A^\vee\otimes_A\mathbf{1}.$$
\end{theorem}

For example, if $G$ is an $L$-locally stably dualizable group in the sense of Rognes (\cite{stdual}), then the spectrum $C^*(G)$ of maps from $G$ to $LS$ is an augmented $E_\infty$-algebra in $LSp$ (the augmentation given by the identity element of $G$), and \cite{stdual} Prop.\ 3.2.3 implies that $C^*(G)$ is in fact a Poincar\'{e} duality algebra, and that the resulting sphere $Sph^{alg}(C^*(G))$ identifies with the spectrum $S^{adG}$ of loc.\ cit.\ Def.\ 2.3.1.

Finally, here is the result about the relationship between the first formalism and the second formalism.  Again, the orbifold definitions will be given later (Section \ref{uniorb}).  The key claim is the third; the other two are just preliminary.

\begin{theorem}\label{unithm}
Let $L$ be a Bousfield localization of spectra, and let $CUO_*$ denote the full subcategory of $Orb_*$ consisting of the $L$-compact pointed orbifolds possessing $L$-unipotent duality.  Then:

\begin{enumerate}
\item $CUO_*$ becomes an $\infty$-category with fibrations by taking the fibrations to be the $L$-unipotent submersions.
\item The association $M\mapsto C^*(M)$ (here, meaning the spectrum of maps from $M$ to $LS$) gives an exact functor (Definition \ref{exact}) from $CUO_*$ to $PD_*^{op}$.
\item The composite
$$K(CUO_*)\overset{[C^*]}{\longrightarrow} K(PD_*^{op})\overset{Sph^{alg}}{\longrightarrow} Pic(LSp)$$
is canonically homotopic to the composite
$$K(CUO_*)\overset{[forget]}{\longrightarrow} K(Orb_*)\overset{Sph^{top}}{\longrightarrow} Pic(LSp).$$
\end{enumerate}
\end{theorem}

The second claim is an Eilenberg-Moore-type result; with Proposition \ref{criterion}, c.f.\ \cite{shipley} Thm.\ 3.1.

Now, before embarking on the proofs, we say a word about the methods that will be employed.  Both mechanisms for producing spheres will factor through a theory of $\infty$-categories of ``sheaves'' or ``representations'' associated to the objects under consideration:  in the first case, to each orbifold $M$ is associated its $\infty$-category $Sh(M)$ of sheaves of spectra, with contravariant ``pullback'' functoriality for maps of orbifolds; in the second case, to each Poincar\'{e} duality algebra $A$ is associated its $\infty$-category $Mod_A$ of modules, with covariant ``base-change'' functoriality for maps of algebras.

The first mechanism will be described in terms of these $\infty$-categories of sheaves as follows.  The the pullback $p^*\colon Sh(M)\rightarrow Sh(*)$ along the projection $p\colon M\rightarrow *$ has a left adjoint $p_\natural$, and the pullback $e^*\colon Sh(M)\rightarrow Sh(*)$ along the inclusion $e\colon *\rightarrow M$ of the point has a right adjoint $e_*$ which preserves all colimits.  The composition
$$p_\natural e_*\colon Sh(*)\rightarrow Sh(*)$$
is therefore a colimit-preserving functor from $Sp$ to itself.  Every such functor is equivalent to smashing with some fixed and uniquely specified spectrum, and here this spectrum is the desired $Sph^{top}(M)$.  (This method of identifying $Sph^{top}(M)$, fundamental for our formalism here, was based on a suggestion of Nick Rozenblyum in a related context.)

Thus the association $M\mapsto Sph^{top}(M)$ is reinterpreted as $M\mapsto p_\natural e_*$, the advantage being that now the desired multiplicative compatibility of $Sph^{top}$ in fiber sequences $M'\overset{i}{\longrightarrow} M\overset{f}{\longrightarrow} M''$ with $f$ a submersion can be seen as having its origin in a natural equivalence $f_\natural i_*\simeq e''_*p'_\natural$ of functors associated to the pullback square
$$\xymatrix{M'\ar[r]^{i}\ar[d]^{p'} & M\ar[d]^{f} \\
\ast\ar[r]^{e''} & M'' }$$
of orbifolds.  Indeed, such an equivalence gives rise to an identification of the functor $p_\natural e_*$ with the composition $(p''_\natural e''_*)\circ (p'_\natural e'_*)$, and composition and smash product agree for colimit-preserving functors $Sp\rightarrow Sp$.

An important point is that this equivalence $f_\natural i_*\simeq e''_*p'_\natural$ does not need to be explicitly produced.  In fact, the equivalence $i^*f^*\simeq p'^*e''^*$ coming just from the commutativity of the above square formally determines a map $p'_\natural i^*\rightarrow e''^*f_\natural$; this map turns out to be an equivalence, and its inverse again formally determines a map $f_\natural i_*\rightarrow e''_*p'_\natural$, and it is this precise map that is claimed to be an equivalence.  By the same token, these equivalences also automatically satisfy a number of compatibilities, for instance bearing on what happens when one composes similar such squares horizontally or vertically.

Thus it turns out that the whole structure of the first mechanism for producing spheres already exists in the (symmetric monoidal) functoriality of the association $M\mapsto Sh(M)$; to obtain Theorem \ref{stacksphere} one needs only verify that certain maps are invertible (namely, we need invertibility of the functor $p_\natural e_*$ as well as of the two natural transformations described in the above paragraph).

The second mechanism (Theorem \ref{algsphere}) is produced in exactly the same way starting from the theory of $\infty$-categories of modules over Poincar\'{e} duality algebras.  Furthermore, this categorical perspective also makes transparent the connection between the two mechanisms (Theorem \ref{unithm}).  Indeed, there is a natural fully faithful embedding $Mod_{C^*(M)}\rightarrow Sh(M)$, coming from the fact that the spectrum of endomorphisms of the unit object of $Sh(M)$ identifies with $C^*(M)$; and while this embedding is not a priori compatible with the functors $p_\natural$ and $e_*$, in the necessary cases it is compatible with their left adjoints, and these determine $Sph^{top}(M)$ just as well.

Now we begin with our orbifold definitions.

\subsection{Orbifold definitions}\label{orbdef}

We will take a toposistic approach, so our orbifolds will be certain topoi.  Even though classical topoi suffice for this purpose, we will instead use $\infty$-topoi, because these more immediately connect to their $\infty$-categories of sheaves of spectra.

\begin{definition} Let $f\colon X\rightarrow Y$ be a map of $\infty$-topoi.  We say that $f$ is finite if, locally on $Y$, it factors as a closed immersion (\cite{htt} Section 7.3.2) followed by projection off a finite set.
\end{definition}

The class of finite maps is closed under composition and base change, and a map between finite maps is also finite.

\begin{definition} We say that an $\infty$-topos $X$ is an orbifold if the following conditions are satisfied:
\begin{enumerate}
\item $X$ locally equivalent to some $\mathbb{R}^n$ (more properly, to the $\infty$-category of sheaves of spaces on some $\mathbb{R}^n$);
\item The diagonal $X\rightarrow X\times X$ is finite.
\end{enumerate}
\end{definition}

Here is the basic example of an orbifold:

\begin{example}
If $M$ is a topological manifold and $G$ is a (discrete) group acting properly discontinuously on $M$, then the quotient $M/G$ (i.e.\ $G$-equivariant sheaves of spaces on $M$) is an orbifold.
\end{example}

(We adopt the convention that topological manifolds are assumed Hausdorff.)

\begin{definition}
We say that a map of orbifolds $X\rightarrow Y$ is a submersion if, locally on $X$, as projection off some $\mathbb{R}^n$ followed by an \'{e}tale map (\cite{htt} Section 6.3.5) to $Y$.
\end{definition}

For example, if $M$ is any orbifold, then the map $p\colon M\rightarrow *$ is a submersion; also, the class of submersions is closed under composition and pullback.

Now we get to work on Theorem \ref{stacksphere}.  The fact that the $\infty$-category $Orb_*$ of pointed orbifolds becomes an $\infty$-category with fibrations where the fibrations are the submsersions is clear from the definitions, so it's all about producing the map of spectra
$$Sph^{top}\colon K(Orb_*)\rightarrow Pic(Sp)$$
and checking what it does on pointed manifolds.  For this, we will pass from orbifolds to their $\infty$-categories of sheaves of spectra.  In the $\infty$-topos language, this is accomplished by means of the functor
$$Sh\colon Orb\rightarrow StPr^{op}$$
which assigns to each orbifold the stabilization (c.f.\ \cite{ha} Ex.\ 6.3.1.22) of its underlying presentable $\infty$-category, with the contravariant pullback functoriality.

Here $StPr$ denotes the $\infty$-category of stable presentable $\infty$-categories, with morphisms the colimit preserving functors.  We'll have to get pretty warm and fuzzy with this $\infty$-category of $\infty$-categories, so here are a few definitions and results.

\subsection{Adjointability in $StPr$}\label{adjstpr}

\begin{definition}\label{adjoint}
Let $f\colon \mathcal{M}\rightarrow\mathcal{N}$ be a map in $StPr$.  We say that $f$ is left adjointable if it has a left adjoint, and we say that $f$ is right adjointable if its right adjoint preserves colimits.
\end{definition}

The point is that in both cases the resulting adjoint can also be viewed as a map in $StPr$.

For the next lemma, recall (\cite{ha} Section 6.3, esp.\ Prop.\ 6.3.2.15) that $StPr$ carries a canonical symmetric monoidal structure $\otimes$ having the property that, for $\mathcal{M},\mathcal{N}$ and $\mathcal{P}$ in $StPr$, the $\infty$-category of colimit preserving functors $\mathcal{M}\otimes \mathcal{N}\rightarrow \mathcal{P}$ is equivalent to the $\infty$-category of functors $\mathcal{M}\times\mathcal{N}\rightarrow \mathcal{P}$ which preserve colimits in each variable separately.  This tensor structure $\otimes$ on $StPr$, used in \cite{ha} to produce and develop the smash product of spectra, will be our main tool for defining the $E_\infty$-structures underlying Theorems \ref{stacksphere} and \ref{algsphere}.

\begin{lemma}\label{adjtens}
The class of left adjointable morphisms in $StPr$ is closed under $\otimes$.  Same with ``right'' instead of ``left''.
\end{lemma}

\begin{proof}
We only discuss ``left'', since ``right'' is just the same.  Suppose that $f$ is left adjointable with left adjoint $g$, and ditto $f'$ with left adjoint $g'$.  Then there are natural transformations $id\rightarrow fg$ and $gf\rightarrow id$ such that the compositions $f\rightarrow fgf\rightarrow f$ and $g\rightarrow gfg\rightarrow g$ are homotopic to the identity, and ditto for $f'$ and $g'$.  But all these functors preserve colimits, so the universal property of the tensor structure on $StPr$ implies that we can tensor these relations together, showing that $g\otimes g'$ is left adjoint to $f\otimes f'$.  Thus a fortiori $f\otimes f'$ is left adjointable, as desired.
\end{proof}

In just a second we will prove a strengthened replacement of this lemma (Proposition \ref{cohprop}); but first another definition, concerning how adjointable maps interact with commutative squares (c.f.\ \cite{ha} Def.\ 6.2.3.13):

\begin{definition}\label{adjsq} Let $\sigma\colon \Delta^1\times\Delta^1\rightarrow StPr$ be a map, parametrizing a diagram
$$\xymatrix{
\mathcal{M'}\ar[r]^-{g'}\ar[d]^-{f'} & \mathcal{M}\ar[d]^-{f}\\
\mathcal{N'}\ar[r]^-{g} & \mathcal{N}
}$$
in $StPr$ which commutes up to a specified homotopy $\alpha\colon gf'\simeq fg'$.

We say that $\sigma$ is vertically left adjointable if $f$ and $f'$ are left adjointable, say with left adjoints $F$ and $F'$ respectively, and the canonical map $Fg\rightarrow g'F'$ given by
$$Fg\rightarrow Fgf'F'\overset{\alpha}{\simeq} Ffg'F'\rightarrow g'F'$$
is an equivalence.

Similarly we can define what it means for $\sigma$ to be vertically right adjointable, or horizontally right or left adjointable.
\end{definition}

Thus if $\sigma$ is vertically left adjointable, we get a new commutative square in $StPr$ where the vertical maps are replaced by their left adjoints, and similarly for the three other variations.

We caution that, for convenience, the modifiers ``vertically'' and ``horizontally'' will always refer to the manner in which the square is drawn on the page, and don't necessarily correlate with the first and second coordinates of $\Delta^1\times\Delta^1$.

Now we turn to a refined version of Lemma \ref{adjtens}, asserting not only that passing to left adjoints of left adjointable maps is compatible with $\otimes$, but also that this passage can be made coherent with respect to both $\otimes$ and composition in $StPr$ (simultaneously, even).  Furthermore, this coherence needs to be encoded not just for morphisms in $StPr$, but also for squares in $StPr$ as in Definition \ref{adjsq}.  This may seem like an overwhelming amount of coherence data to specify, but the following definition (an immediate analog of \cite{ha} Def.\ 6.2.3.16) permits an efficient encoding.

\begin{definition}\label{adjcoh}
Let $K$ be a simplicial set.  We denote by $Fun^{LAd}(K,StPr)$ the $\infty$-subcategory of $Fun(K,StPr)$ whose objects are those $F\colon K\rightarrow StPr$ such that $F(i)\rightarrow F(j)$ is left adjointable for all maps $i\rightarrow j$ in $K$, and whose morphisms are those natural transformations $F\rightarrow G$ for which the square
$$\xymatrix{
F(i)\ar[r]\ar[d] & F(j)\ar[d]\\
G(i)\ar[r] & G(j)
}$$
is horizontally left adjointable for all $i\rightarrow j$ in $K$.

We can similarly define the $\infty$-subcategory $Fun^{RAd}(K,StPr)$ of $Fun(K,StPr)$.
\end{definition}

Note that, by virtue of Lemma \ref{adjtens} (and its analog for squares, which can be proved in the same way), these $\infty$-categories $Fun^{LAd}(K,StPr)$ and $Fun^{RAd}(K,StPr)$ inherit the tensor product from $StPr$.  Furthermore, they are contravariantly functorial in $K$ via pullback.

Now here is the relevant coherence statement.

\begin{proposition}\label{cohprop}
Let $K$ be a simplicial set.  Then there is a canonical symmetric monoidal equivalence
$$Fun^{LAd}(K,StPr)\simeq Fun^{RAd}(K^{op},StPr),$$
functorial in $K$ for pullbacks.
\end{proposition}

The proof will show that on the level of objects the asserted equivalence is given by passing to left adjoints of all the maps in the corresponding $K$-diagram in $StPr$, and that on maps it does similarly for squares; thus this statement does encode the required coherences.

\begin{proof}
This is almost contained in \cite{ha}; it suffices to combine the material of \cite{ha} Section 6.2.3 with that of \cite{ha} Section 6.3.  More precisely, let $\mathbf{O}_K$ denote the fibrant simplicial colored operad defined as follows:
\begin{enumerate}
\item An object of $\mathbf{O}_K$ is a map of simplicial sets $p\colon X\rightarrow K$ with the following properties:
\begin{enumerate}
\item $p$ is both a Cartesian fibration and a coCartesian fibration;
\item For each $k$ in $K_0$, the fiber $\infty$-category $X_k$ is stable and presentable;
\item For each $k\rightarrow k'$ in $K_1$, the functor $X_{k'}\rightarrow X_{k}$ induced by the Cartesian condition and the functor $X_k\rightarrow X_{k'}$ induced by the coCartesian condition both preserve all colimits.
\end{enumerate}
\item For a finite collection $\{p_i\colon X_i\rightarrow K\}_{i\in I}$ of objects of $\mathbf{O}_K$ with fiber product $p\colon X\rightarrow K$ and an object $q\colon Y\rightarrow K$ of $\mathbf{O}_K$, the simplicial set of multi-maps $\{p_i\}_{i\in I}\rightarrow q$ in $\mathbf{O}_K$ is the sub-simplicial set of $Map_K(X,Y)$ spanned by those $f\colon X\rightarrow Y$ over $K$ with the following properties:
\begin{enumerate}
\item $f$ sends $p$-Cartesian edges to $q$-Cartesian edges and $p$-coCartesian edges to $q$-coCartesian edges;
\item For each $k$ in $K_0$, the induced functor on fibers $f_k\colon \prod_{i\in I}(X_i)_k\rightarrow Y_k$ preserves all colimits in each variable separately.
\end{enumerate}
\end{enumerate}
Then via Cartesian straightening, the operadic nerve (\cite{ha} Prop.\ 2.1.1.27) of $\mathbf{O}_K$ identifies with $Fun^{RAd}(K^{op},StPr)$; and via coCartesian straightening, it identifies with $Fun^{LAd}(K,StPr)$ (c.f.\ \cite{ha} Cor.\ 6.3.1.4).  This gives the required equivalence of symmetric monoidal $\infty$-categories, and the functoriality in $K$ is evident from the construction.\end{proof}

\subsection{Orbifold-to-sheaf lemmas}\label{orb2shf}

With this $StPr$ background at hand, we now prove the lemmas that glean what we need from the functor $Sh\colon Orb\rightarrow StPr^{op}$ of sheaves of spectra on orbifolds.

\begin{lemma}\label{shprod}
The functor $Sh\colon Orb\rightarrow StPr^{op}$ is symmetric monoidal for cartesian product on the source and $\otimes$ on the target.
\end{lemma}

\begin{proof}
This is a general fact about $\infty$-topoi:  combine \cite{ha} Remark 6.3.1.18 (the analogous statement for sheaves of spaces) with \cite{ha} Prop.\ 6.3.2.15 (implying that stabilization is symmetric monoidal).
\end{proof}

Tracing things through, we see that the claimed equivalence $Sh(X)\otimes Sh(Y)\simeq Sh(X\times Y)$ corresponds to the external smash product of sheaves of spectra.

The next lemma says what's important about submersions for our present purposes.

\begin{lemma}\label{subadj} Suppose that $f\colon M\rightarrow N$ is a submersion of orbifolds, and that $f'\colon M'\rightarrow N'$ is a pullback $f$ via some map $g\colon N'\rightarrow N$.  Then the square
$$\xymatrix{ Sh(M') & Sh(M)\ar[l]_{g'^*}\\
Sh(N')\ar[u]^-{f'^*} & Sh(N)\ar[l]_{g^*}\ar[u]^-{f^*}}$$
is vertically left adjointable.
\end{lemma}

\begin{proof}
The claim is local on $M$, so we can reduce to two separate cases: one where $f$ is projection off some $\mathbb{R}^n$, and the other where $f$ is \'{e}tale.  The second case is trivial.  In the first case, by Lemma \ref{shprod} the above square identifies with the tensor product of the following two squares:
$$\xymatrix{Sh(N') & Sh(N)\ar[l]_{g^*} & & & Sh(\mathbb{R}^n) & Sh(\mathbb{R}^n)\ar@{=}[l] \\
Sh(N')\ar@{=}[u] & Sh(N)\ar[l]_{g^*}\ar@{=}[u] & & & Sh(*)\ar[u]^-{p^*} & Sh(*)\ar@{=}[l]\ar[u]^-{p^*}}$$ 
Thus by Proposition \ref{cohprop} we reduce to showing that $p^*$ is left adjointable.  But $\mathbb{R}^n$ is locally of constant shape (\cite{ha} Remark A.1.4), so this follows from \cite{ha} Proposition A.1.8.
\end{proof}

When $f$ is a submersion, we will denote the above-guaranteed left adjoint to $f^*$ by $f_\natural$.  Note (\cite{ha} Rmk.\ A.1.10) that if $M$ is an orbifold and $p\colon M\rightarrow\ast$ is the projection to the point, then we can identify
$$p_\natural \mathbf{1}\simeq \Sigma^\infty_+|M|,$$
where $\mathbf{1}=p^*(S)$ denotes the constant sheaf with values in the sphere spectrum and on the right we mean the suspension spectrum of the homotopy type (shape) of $M$.  This is indeed more ``natural'' than $p_*\mathbf{1}\simeq\underline{Map}(\Sigma^\infty_+|M|,S)$, justifying the choice of subscript in the notation $f_\natural$.

As their proofs show, the above two lemmas hold even on the level of sheaves of spaces.  In contrast, the following result requires stabilization.

\begin{lemma}\label{doubleadj} Let $i\colon N'\rightarrow N$ be a finite map of orbifolds, and let $i'\colon M'\rightarrow M$ be a pullback of $i$ by a submersion $f\colon M\rightarrow N$.  Then the square
$$\xymatrix{
Sh(M')\ar[d]^-{f'_\natural} & Sh(M)\ar[l]_-{i'^*}\ar[d]^-{f_\natural} \\
Sh(N') & Sh(N)\ar[l]_-{i^*}
}$$
(produced by the vertical left adjointability claim in  Lemma \ref{subadj}) is horizontally right adjointable.\end{lemma}
\begin{proof}
The claim is local on $N$ (using the previous lemma), so we can reduce to two separate cases:  one where $i$ is a closed immersion, and the other where $i$ is projection off a finite set $S$.  For the first case, if $j$ denotes the inclusion of the open complement, then $i^*$ and $j^*$ together detect equivalences, so we can check the desired conclusions after applying them; but then the equations $j^*i_*=0$ and $i^*i_*=id$, as well as the same ones for the base-changes of $i$ and $j$ by $f$, make the claims trivial.  In the second case, our square is the tensor product of the following two squares:
$$\xymatrix{ Sh(S) & Sh(*)\ar[l]_-{p^*}\ar@{=}[d]  & & & Sh(M)\ar[d]^-{f_\natural}\ar@{=}[r] & Sh(M)\ar[d]^-{f_\natural} \\ 
Sh(S)\ar@{=}[u] & Sh(*)\ar[l]_-{p^*} & & & Sh(N)\ar@{=}[r] & Sh(N) }$$
Thus Proposition \ref{cohprop} reduces us to showing that $p^*$ is right adjointable.  But its right adjoint $p_*$ is just given by taking finite products, and these agree with finite coproducts in the stable setting, so $p_*$ does preserve colimits, as desired.\end{proof}

Recall that we are particularly interested in the functor $p_\natural e_*\colon Sh(*)\rightarrow Sh(*)$ associated to a pointed orbifold $*\overset{e}{\longrightarrow} M\overset{p}{\longrightarrow} *$.  Since this is a colimit-preserving functor from spectra to itself, it is uniquely determined by its value on the sphere spectrum $S$:  we naturally have $p_\natural e_*X\simeq X\wedge p_\natural e_*S$ for all $X\in Sp$.  Thus the next lemma calculates $p_\natural e_*$ when $M$ is a pointed manifold:

\begin{lemma}\label{mansph}  Let $e\colon *\rightarrow M$ be a pointed topological manifold, and let $p\colon M\rightarrow *$.  Then there is a natural equivalence of spectra
$$p_\natural e_*S\simeq \Sigma^\infty |M|/|M-\ast|.$$
\end{lemma}
\begin{proof}
Let $j$ denote the inclusion of the open complement $M-\ast$ in $M$.  Evaluating the cofiber sequence
$$j_\natural j^*\rightarrow id\rightarrow e_*e^*$$
on $p^*S$ and then applying $p_\natural$ gives the cofiber sequence
$$\Sigma^\infty_+|M-\ast|\rightarrow \Sigma^\infty_+ |M|\rightarrow p_\natural e_* S,$$
yielding the desired.
\end{proof}

In particular, since the space $|M|/|M-\ast|$ is equivalent to a sphere, we see that the functor $p_\natural e_*$ is an equivalence when $M$ is a pointed manifold.  But in fact this holds for an arbitrary pointed orbifold:

\begin{corollary}\label{invertible}
Let $e\colon *\rightarrow M$ be a pointed orbifold, and let $p\colon M\rightarrow *$ denote the projection.  Then the functor $p_\natural e_*\colon Sh(*)\rightarrow Sh(*)$ is an equivalence.
\end{corollary}
\begin{proof}
Choose an etale map $\mathbb{R}^n\rightarrow M$ whose pullback $F$ by $e$ is nonempty.  Then $F$ is both $0$-localic (being finite over the $0$-localic $\mathbb{R}^n$) and a slice topos of the $\infty$-category of spaces $\mathcal{S}$ (being etale over $*$); thus $F$ is equivalent to a (nonempty) discrete set.

Now, choose a point of $F$.  Then Lemma \ref{doubleadj} applied to the pullback square
$$\xymatrix{ F\ar[r]\ar[d] & \mathbb{R}^n\ar[d] \\
\ast\ar[r] & M }$$
shows that $p_\natural e_*$, when precomposed with the analogous composition for $*\rightarrow F$, identifies with the analogous composition for $*\rightarrow \mathbb{R}^n$; thus we are reduced to $M=F$ and $M=\mathbb{R}^n$.  But in both these cases $M$ is a manifold, so the previous lemma lets us conclude.
\end{proof}

The above argument is ad hoc; the later Lemma \ref{duallemma} gives an alternate, more systematic approach.

\subsection{Proof of Theorem \ref{stacksphere}}\label{stackpf}

Now we have enough lemmas to prove Theorem \ref{stacksphere}.  Thus let us produce a map of spectra
$$Sph^{top}\colon K(Orb_*)\rightarrow Pic(Sp).$$
By group completion, it is enough to give an $E_\infty$-map $\Omega |S_\bullet Orb_*|\rightarrow Inv(Sp)$.  In turn, by applying $\Omega$ it is enough to give an $E_\infty$-map $|S_\bullet Orb_*|\rightarrow StPr^\sim$ (where the latter gets its $E_\infty$-structure from the tensor product).  Finally, giving such a map $|S_\bullet Orb_*|\rightarrow StPr^\sim$ is equivalent to giving a system of $E_\infty$-maps
$$S_I Orb_*\rightarrow StPr^\sim,$$
functorial in $I\in\Delta^{op}$; so this is what we will do.

Recall (Section \ref{sdot}) that $S_IOrb_*$ is the $E_\infty$-space of objects in a certain full subcategory of $Fun(Ar(I),Orb_*)$.  We start with the observation that $Ar(I)$ is canonically built out of rectangles.  This is ``obvious'' from how it is drawn in the plane, but formally we can say that there is a category $AR(I)$, a functor $AR(I)\rightarrow \Delta\times \Delta$, and an isomorphism of simplicial sets
$$\operatorname{colim}_{AR(I)}\Delta^n\times\Delta^m\simeq Ar(I),$$
all of this functorial in $I$.  Indeed, we can let $AR(I)$ be the Grothendieck construction on the bisimplicial set $\Delta^{op}\times\Delta^{op}\rightarrow Sets$ given by
$$([n],[m])\mapsto \{f\colon [n]\rightarrow I,g\colon [m]\rightarrow I\text{ }\mid\text{ }f(x)\geq g(y) \text{ for all }x\in [n],y\in [m]\}.$$

Thus we can canonically write $Fun(Ar(I),Orb_*)$ as a limit of the $\infty$-categories $Fun(\Delta^n\times\Delta^m,Orb_*)$, and similarly $S_IOrb_*$ is canonically a limit (over $a\in AR(I)$, say) of the space of objects $S_a$ in certain full subcategory of $Fun(\Delta^n\times\Delta^m,Orb_*)$.  Then Lemma \ref{shprod} shows that applying $Sh$ gives an $E_\infty$-map from $S_a$ to the space of objects $S'_a$ in a certain full subcategory of $Fun((\Delta^n)^{op}\times(\Delta^m)^{op},StPr)$.  But by Lemma \ref{subadj}, everything in $S'_a$ has the property that all of the little squares in the corresponding $StPr$-diagram are left adjointable in the $m$-direction; thus we can use Proposition \ref{cohprop} to coherently pass to left adjoints in the $m$-direction in these diagrams, thereby producing a map from $S'_a$ to the space of  objects $S''_a$ in a certain full subcategory of $Fun((\Delta^n)^{op}\times\Delta^m,StPr)$, this being functorial in $\Delta^n$ and $\Delta^m$ and hence in $a$.  And by the same token, because of Lemma \ref{doubleadj} we can further map to the space of objects $S'''_a$ in a certain full subcategory of $Fun(\Delta^n\times \Delta^m,StPr)$.

Thus, taking the limit of these maps $S_a\rightarrow S'''_a$ over $a\in AR(I)$, we have produced an $E_\infty$-map
$$S_IOrb_*\rightarrow Fun(Ar(I),StPr)^\sim$$
functorially in $I\in\Delta^{op}$.   Just to summarize, this map encodes the procedure of passing to sheaves of spectra in an $Ar(I)$-diagram in $Orb_*$, then replacing the vertical maps by their left adjoints and the horizontal maps by their right adjoints.  Now we further map to $Fun(I,StPr)^\sim$ by restricting along the inclusion $I\rightarrow Ar(I)$ of the identity morphisms.  Then by Lemma \ref{invertible}, every arrow in $I$ goes to an equivalence in $StPr$ under each of the diagrams $I\rightarrow StPr$ in the image of this map $S_IOrb_*\rightarrow Fun(I,StPr)^\sim$; thus this latter canonically factors through an $E_\infty$-map
$$S_IOrb_*\rightarrow Fun(|I|,StPr)^\sim$$
where $|\cdot |$ denotes geometric realization, and evidently this is still functorial in $I\in\Delta^{op}$.  However, the geometric realization of any simplex $I$ is contractible, so this last map gives exactly what we wanted.

Thus we have produced the desired map of spectra $Sph^{top}\colon K(Orb_*)\rightarrow Pic(Sp)$.  Then all that rests of Theorem \ref{stacksphere} is the ancillary claim concerning the value of $Sph^{top}$ on pointed manifolds;  but this follows directly from Lemma \ref{mansph}, bearing in mind that the $E_\infty$-ness of the resulting identification is a consequence of Proposition \ref{cohprop} and the fact that all the functors used in the proof of Lemma \ref{mansph} preserve colimits.

This finishes the proof; but let us also record the following consequence, which makes explicit reference to our chosen formalism for encoding $M\mapsto Sph^{top}(M)$:

\begin{proposition}\label{bgshf}
Let $Gp^\sim$ denote the $E_\infty$-space of finite groups under cartesian product.  Then the following two trivializations of the map $G\mapsto Sph^{top}(BG)$ from $Gp^\sim$ to $Inv(Sp)$ coincide:
\begin{enumerate}
\item First, the one induced by using the fiber sequences $G\rightarrow *\rightarrow BG$ and $*\rightarrow *\rightarrow G$ to trivialize $[BG]$ in $\Omega^\infty K(Orb_*)$ and hence $Sph^{top}(BG)$ in $Inv(Sp)$;
\item Second, letting $*\overset{e}{\longrightarrow} BG\overset{p}{\longrightarrow} *$ denote the important maps, the one given by the equivalence of spectra
$$Sph^{top}(BG)=p_\natural e_*(S)=\left(\prod_GS\right)_G\simeq \left(\vee_G S\right)_G\simeq S,$$
where $G$ in a subscript stands for (homotopy) $G$-orbits, and the middle equivalence comes from the identification of finite coproducts and finite products in the stable setting (manifested here in an equivlanece between $e_\natural$ and $e_*$).
\end{enumerate}
\end{proposition}

\begin{proof} This is a simple unwinding: the point is that the equivalence $e_*\simeq e_\natural$ used in the second trivialization is equivalent to the one gotten by the following procedure:  first apply Lemma \ref{doubleadj} to the pullback of $i\colon *\rightarrow G$ by itself to obtain a canonical equivalence $i_\natural\simeq i_*$, then apply Lemma \ref{doubleadj} to the pullback of $e\colon *\rightarrow BG$ by itself and evaluate the resulting right adjointed square on $i_*\simeq i_\natural$ to deduce the claimed equivalence $e_*\simeq e_\natural$.\end{proof}

\subsection{Spheres from Poincar\'{e} duality algebras}\label{sphalg}

Now we mimic the above to prove Theorem \ref{algsphere}.  A difference is that we fix a Bousfield localization $L$ of spectra, and, instead of working in the $\infty$-category $StPr$ of stable presentable $\infty$-categories (\ref{adjstpr}), we work in the $\infty$-category $LPr$ of $L$-local stable presentable $\infty$-categories, meaning the module objects over $LSp$ in $StPr$.  Since $LSp$ is idempotent (\ref{loc}), this $\infty$-category $LPr$ is a full subcategory of $StPr$ and inherits the symmetric monoidal tensor product as such; one then sees that all the definitions and results of Section \ref{adjstpr} (on adjointability of morphisms and squares) carry over to this setting.

Those preliminaries spoken, we start with the lemmas.  The $\infty$-category of Poincar\'{e} duality algebras in $LSp$ (Def.\ \ref{pdalg}) will be denoted by $PD$; thus the Bousfield localization $L$ is implicit.

\begin{lemma}\label{modprod}
The functor $Mod\colon PD\rightarrow LPr$ which sends a PD-algebra $A$ in $LSp$ to its $\infty$-category $Mod_A$ of modules is symmetric monoidal for coproduct on the source and tensor product on the target.
\end{lemma}
In fact, the Poincar\'{e} duality hypothesis plays no role here, and the same result holds with $LSp$ replaced by a fairly general symmteric monoidal $\infty$-category:
\begin{proof}
See \cite{ha} Remark 6.3.5.15.
\end{proof}

For a map $f\colon A\rightarrow B$ of $E_\infty$-agebras in $LSp$, we let $f^*\colon Mod_A\rightarrow Mod_B$ denote the functor of co-base change, i.e.\ $f^*(M)=M\otimes_AB$.

\begin{lemma} Let $f\colon A\rightarrow B$ be a Poincar\'{e} duality map of $E_\infty$-algebras in $LSp$, and let $f'\colon A'\rightarrow B'$ be a co-base change of $f$ via some map $g\colon A\rightarrow A'$.  Then the square
$$\xymatrix{Mod_{B'} & Mod_{B}\ar[l]_-{g'^*} \\
Mod_{A'}\ar[u]^-{f'^*} & Mod_A\ar[l]_-{g^*}\ar[u]^-{f^*} }$$
is vertically left adjointable, and the resulting square
$$\xymatrix{Mod_{B'}\ar[d]^-{f'_\natural} & Mod_B\ar[l]_-{g'^*}\ar[d]^-{f_\natural} \\
Mod_{A'} & Mod_A\ar[l]_-{g^*} }$$
is horizontally right adjointable.
\end{lemma}

In fact, we will only use the first condition of the definition (\ref{pdalg}) of a Poincar\'{e} duality map, namely that $B$ is dualizable as an $A$-module.

\begin{proof}
Indeed, the left adjoint of $f^*$ can be explicitly written down as $f_\natural(N)=B^\vee\otimes_B N$, where $B^\vee$ denotes the $A$-dual of $B$, and similarly for $f'_\natural$.  Then to check that the natural map $f'_\natural g'^*\rightarrow g^*f_\natural$ is an equivalence, it suffices to verify this on $B\in Mod_B$ (because all these functors preserve colimits, and $B$ generates $Mod_B$), and there it follows from the fact that $g^*$, being symmetric monoidal, preserves duals.  This verifies the vertical left adjointability claim.  For the further horizontal right adjointability claim, we note first that the right adjoint $g_*$ of $g^*$ is given by forgetting the $A'$-module structure and certainly preserves colimits, and similarly for $g'^*$.  Then again to verify that the map $f_\natural g'_*\rightarrow g'_*f'_\natural$ is an equivalence it suffices to check on $B'$, where it is an equivalence for essentially the same reason as above.
\end{proof}

A more satisfying proof could be given by extending the adjointability formalism (\ref{adjstpr}) to the setting of module categories over $Mod_A$, then arguing as in Lemmas \ref{subadj} and \ref{doubleadj}, decomposing the above square of pullbacks into a tensor product (over $Mod_A$) of squares which exist in only one direction.

\begin{lemma} Let $e\colon A\rightarrow \mathbf{1}$ be an augmented Poincar\'{e} duality algebra in $LSp$, and let $p\colon \mathbf{1}\rightarrow A$ be the unit map.  Then we have
$$p_\natural e_*\mathbf{1}\simeq \mathbf{1}\otimes_AA^\vee,$$
and (hence) $p_\natural e_*\colon LSp\rightarrow LSp$ is an equivalence.\end{lemma}

Here $\mathbf{1}\in LSp$ denotes the $L$-local sphere $LS$, which is the unit object of $LSp$.

\begin{proof}
The first claim is immediate from the identification of the functors $p_\natural$ and $e_*$ given in the above proof, and then the second claim follows from the second condition in the definition of a Poincar\'{e} duality algebra (\ref{pdalg}), which implies that $\mathbf{1}\otimes_AA^\vee\in LSp$ is invertible.
\end{proof}

With these lemmas at hand, the proof of Theorem \ref{algsphere} is exactly the same as that of Theorem \ref{stacksphere} (given in Section \ref{stackpf}), except one replaces $Orb_*$ by $PD_*^{op}$, $Sh$ by $Mod$, $Sp$ by $LSp$, $StPr$ by $LPr$, and the references to the orbifold lemmas by references to the corresponding above PD-algebra lemmas.  We say no more about it.
 
\subsection{Comparing the orbifoldic spheres with the algebraic ones}

Finally, we turn to the proof of Theorem \ref{unithm}.  Again $L$ will be a fixed Bousfield localization of spectra, and from now on we will replace $Sh$ with $Sh\otimes LSp$, so that $Sh$ is now a functor from $Orb$ to $LPr^{op}$.  Thus we are considering sheaves of $L$-local spectra on orbifolds, instead of sheaves of plain spectra.  Note that, because of Proposition \ref{cohprop} and the idempotency of $LSp$, everything in Sections \ref{orb2shf} and \ref{stackpf} still goes through with this replacement, except that the map of spectra $Sph^{top}\colon K(Orb_*)\rightarrow Pic(Sp)$ produced in Section \ref{stackpf} gets replaced by its composite with $L$, i.e.\
$$LSph^{top}\colon K(Orb_*)\rightarrow Pic(Sp)\rightarrow Pic(LSp).$$

In fact, since we will be in this Bousfield-local setting for the rest of the appendix, let us remark once and for all that, although we do fix a localization $L$ for the purposes of discussion, every notion we consider commutes with further localization (\ref{bousfact}).  For example, if an orbifold $M$ is $L$-compact (Definition \ref{lecpct}) and $L'$ is a further localization of $L$, then $M$ is also $L'$-compact; or again the $L'$-dualizing sheaf of $M$ (Definition \ref{dualdef}) identifies with the $L'$-localization of the $L$-dualizing sheaf of $M$, etc.  This will always follow immediately from Proposition \ref{cohprop}, since passage to a further localization $L'$ is effectuated by tensoring with $L'Sp$, and we will only ever use functors that preserve colimits.  This same remark also applies to the previous section on Poincar\'{e} duality algebras.

\subsubsection{Compactness and duality for orbifolds}\label{cpct}

We start with some definitions and results related to compactness and duality.  The eventual point will be that, in the compact case, we have a different access to the crucial functor $p_\natural e_*$:  it is left adjointable, its left adjoint being gotten by composing the double left-adjoint of $p^*$ (which exists in the compact case, as we'll see) with $e^*$.  Thus, under the extra hypothesis of compactness, we can let $e$ rest and make $p$ bear twice the burden.

\begin{definition}\label{lecpct}
An orbifold $M$ is called $L$-compact if $p^*$ is right adjointable, where $p\colon M\rightarrow *$.
\end{definition}

Recall that the right adjoint to $p^*$ is denoted $p_*$; thus $M$ being $L$-compact means that $p_*$ also preserves colimits.

For instance, every compact manifold is an $L$-compact orbifold for all $L$ (this follows from \cite{htt} Cor.\ 7.3.4.12, which says that in the compact Hausdorff case taking global sections preserves filtered colimits).  Another example is that, for a finite group $G$, the orbifold $BG:=*/G$ is $L$-compact whenever $L$ is a further localization of $\#G$-inversion.  Indeed, when $M=BG$ there is a ``norm'' map $p_\natural\rightarrow p_*$ which is an equivalence after $\#G$-inversion.  Further examples can be generated from these by the fact that the product of two $L$-compact orbifolds is also $L$-compact, as follows from Proposition \ref{cohprop} and Lemma \ref{shprod}.

We also introduce the following relative variant of $L$-compactness. We warn that the definition is ad hoc, and the resulting notion may not be closed under composition.

\begin{definition}
A map $M\rightarrow N$ of orbifolds is called $L$-proper if, locally on $N$, it factors as a finite map followed by projection off an $L$-compact orbifold.
\end{definition}

For instance, when $M$ is $L$-compact any map $M\rightarrow N$ of orbifolds is $L$-proper, as can be seen by factoring $f$ via its graph.

The following result is proved in the same way as Lemma \ref{doubleadj}, which it strengthens; we omit the argument.

\begin{lemma}\label{adjproper}
Let $f\colon M\rightarrow N$ be an $L$-proper map of orbifolds, and $f'\colon M'\rightarrow N'$ a pullback of $f$ by some map $g\colon N'\rightarrow N$.  Then the squares
$$\xymatrix{ Sh(M') & Sh(N')\ar[l]_-{f'^*} & & Sh(M')\ar[d]^-{g'_\natural} & Sh(N')\ar[l]_-{f'^*}\ar[d]^-{g_\natural} \\
Sh(M)\ar[u]^-{g'^*} & Sh(N)\ar[l]_-{f^*}\ar[u]^-{g^*} & & Sh(M) & Sh(N)\ar[l]_-{f^*} }$$
in $LPr$ are horizontally right adjointable, where for the second square we assume that $g$ is a submersion and the square is gotten from Lemma \ref{subadj}.
\end{lemma}

We note the following corollary.

\begin{corollary}\label{extraadj}
Let $X\overset{f}{\longrightarrow} Y\rightarrow M$ be maps of orbifolds, and let $\mathcal{F}\in Sh(M)$ and $\mathcal{G}\in Sh(X)$.  Then:
\begin{enumerate}
\item If $f$ is a submersion, the natural map $f_\natural(\mathcal{F}\wedge\mathcal{G})\rightarrow\mathcal{F}\wedge f_\natural\mathcal{G}$ is an equivalence;
\item If $f$ is $L$-proper then the natural map $\mathcal{F}\wedge f_*\mathcal{G}\rightarrow f_*(\mathcal{F}\wedge\mathcal{G})$ is an equivalence.
\end{enumerate}
Here for clarity we are systematically suppressing all pullbacks from $M$.
\end{corollary}

\begin{proof}
The first claim follows from Lemma \ref{subadj} applied to the pullback of $f\times M$ to $f$, and the second claim follows from Lemma \ref{adjproper} applied to the same.
\end{proof}

Now we move on to duality.

\begin{definition}\label{dualdef}
Let $f\colon M\rightarrow N$ be a submersion of orbifolds.  We define the $L$-dualizing sheaf of $f$ to be
$$\mathbb{D}_{M/N}:=(p_1)_\natural\Delta_*\mathbf{1}\in Sh(M),$$
with $\mathbf{1}\in Sh(M)$ the unit, $\Delta\colon M\rightarrow M\times_N M$ the diagonal, and $p_1\colon M\times_NM\rightarrow M$ the first projection.
\end{definition}

When $N=*$, we will write $\mathbb{D}_M$ for $\mathbb{D}_{M/*}$.  The terminology of ``dualizing sheaf'' will be justified after the following preliminary lemma.
 
\begin{lemma}\label{duallemma} Let $f\colon M\rightarrow N$ be a submersion.  Then:
\begin{enumerate}
\item There is a canonical equivalence $(p_1)_\natural\Delta_*\simeq (-)\wedge\mathbb{D}_{M/N}$ of functors $Sh(M)\rightarrow Sh(M)$;
\item  The sheaf $\mathbb{D}_{M/N}$ is invertible (under smash product).
\end{enumerate}
\end{lemma}
\begin{proof}
Since $p_1$ is a submersion and $\Delta$ is finite, the first claim follows from Corollary \ref{extraadj}.

The second claim is local on $M$, so we can reduce to two cases: one where $f$ is projection off some $\mathbb{R}^n$, and the other where $f$ is \'{e}tale.  In the second case $\Delta$  is, locally on $M$, equivalent to projection off a finite set; thus $\mathbb{D}_{M/N}= \Delta_*(p_1)_\natural\mathbf{1}\simeq\Delta_\natural (p_1)_\natural\mathbf{1}\simeq \mathbf{1}$, certainly invertible.  In the first case, by Lemma \ref{shprod} we can assume that $N=*$.  But then the self-homeomorphism $(x,y)\mapsto (x,x-y)$ of $\mathbb{R}^n\times\mathbb{R}^n$ fixes $p_1$ but transforms $\Delta$ into $\mathbb{R}^n\times(0\rightarrow\mathbb{R}^n)$, and so the argument of Lemma \ref{mansph} identifies $\mathbb{D}_{\mathbb{R}^n}$ with the constant sheaf on $cofib(L\Sigma^\infty_+(\mathbb{R}^n-0)\rightarrow L\Sigma^\infty_+\mathbb{R}^n)\simeq L\Sigma^\infty S^n$, which is also certainly invertible.\end{proof}

An elaboration of the reasoning at the end of the above proof would show that, when $M$ is a topological manifold, $\mathbb{D}_M$ identifies with the $L$-localization of the fiberwise ``one-point compactification'' of the tangent microbundle of $M$.
 
Now we justify the name ``dualizing sheaf''.

\begin{proposition}\label{dualprop}
Let $f\colon M\rightarrow N$ an $L$-proper submersion.  Then there is a canonical equivalence
$$f_\natural(-)\simeq f_*(-\wedge\mathbb{D}_{M/N}).$$
\end{proposition}

Thus in this case $\mathbb{D}_{M/N}$ expresses ``homology'' ($f_\natural$) in terms of ``cohomology'' ($f_*$), and $\mathbb{D}_{M/N}^{-1}$ vice-versa.

\begin{proof}
Applying the second claim of Lemma \ref{adjproper} to the pullback of $M\rightarrow N$ by itself and evaluating the resulting right adjointed square on $\Delta_*\mathcal{F}$ gives
$$f_\natural(\mathcal{F})\simeq f_*(p_1)_\natural\Delta_*(\mathcal{F}),$$
so we conclude by the first claim of the previous lemma.\end{proof}

\subsubsection{First step of the proof of Theorem \ref{unithm}}\label{firststep}

We don't yet have enough terminology to properly interpret the statement of Theorem \ref{unithm} (we're missing the notions of an orbifold ``possessing $L$-unipotent duality'' and of a submersion of orbifolds ``being $L$-unipotent''), but nevertheless we can still explain the first step of its proof using the above material.  In fact, only the following corollary is necessary:

\begin{corollary}\label{doubleleft}
Let $f\colon M\rightarrow N$ be a submersion between $L$-compact orbifolds, and let $f'\colon M'\rightarrow N'$ be a pullback of $f$ by a finite map $i\colon N'\rightarrow N$.  Then the square
$$\xymatrix{
Sh(M')\ar[d]^-{f'_\natural} & Sh(M)\ar[l]_-{i'^*}\ar[d]^-{f_\natural} \\
Sh(N') & Sh(N)\ar[l]_-{i^*}
}$$
(produced by the vertical left adjointability claim in  Lemma \ref{subadj}) is vertically left adjointable.
\end{corollary}
\begin{proof}
Since $M$ is $L$-compact, $f$ is $L$-proper in addition to being a submersion.  Then Proposition \ref{dualprop} lets us directly read off the left adjoint of $f_\natural$ as $\mathcal{G}\mapsto p^*(\mathcal{G})\wedge\mathbb{D}_{M/N}^{-1}$.  Similarly, $f'_\natural$ is also left adjointable.

But then the claim is formally equivalent to Lemma \ref{doubleadj}, by passing to adjoints.
\end{proof}
Now we give the first step of the proof of Theorem \ref{unithm}.  It is a reinterpretation of the restriction of $LSph^{top}\colon K(Orb_*)\rightarrow Pic(LSp)$ to the full $\infty$-subcategory with fibrations $CO_*\subseteq Orb_*$ consisting of the $L$-compact pointed orbifolds.  More properly, we give a reinterpretation of the smash-inverse $(LSph^{top})^{-1}$ restricted to $CO_*$, as follows.  We start by exactly following the construction of $LSph^{top}$ (given in Section \ref{stackpf}), replacing $Orb_*$ with $CO_*$ and $StPr$ with $LPr$, until the point where Lemma \ref{doubleadj} was used to pass to right adjoints in the $n$-direction of certain diagrams $(\Delta^n)^{op}\times\Delta^m\rightarrow LPr$, thereby mapping $S_a''\rightarrow S_a'''$.

At that point we instead use Corollary \ref{doubleleft} to pass again to left adjoints in the $m$-direction, thereby mapping $S_a''$ to the $E_\infty$-space of objects in some full subcategory of $Fun((\Delta^n)^{op}\times(\Delta^m)^{op},LPr)$.  This makes the rest of the story left adjoint to what it was before; and since for invertible maps left adjoints are the same as inverses, it has the effect of replacing $LSph^{top}$ by $(LSph^{top})^{-1}$, as claimed.

We can also make the same modification to the construction of $Sph^{alg}\colon K(PD_*)\rightarrow Pic(LSp)$, this time without restricting to any subcategory of $PD$.  Indeed, we need only check that for a PD-map $f\colon A\rightarrow B$ of $E_\infty$-algebras in $LSp$, the functor $f^*\colon Mod_A\rightarrow Mod_B$ has a double left adjoint.  But its left adjoint $N\mapsto B^\vee\otimes_B N$ (recall that $B^\vee$ denotes the $A$-dual of $B$) differs from its right adjoint $N\mapsto N$ by tensoring with the invertible $B$-module $B^\vee$, so the double left adjoint is given by $M\mapsto (B^\vee)^{-1}\otimes_B N$ and does indeed exist, as claimed.

Thus for purposes of the proof of Theorem \ref{unithm} we can replace $LSph^{top}$ with $(LSph^{top})^{-1}$ and $Sph^{alg}$ with $(Sph^{alg})^{-1}$, the constructions of these being changed as indicated above (use double left adjoints on the fibrations instead of using one left adjoint on the fibrations and one right adjoint on the sections).

\subsubsection{Unipotent sheaves on orbifolds}\label{uniorb}

The next step in the proof of Theorem \ref{unithm} will be to trim the fat off from the above construction of $(LSph^{top})^{-1}\colon K(CO_*)\rightarrow Pic(LSp)$.  The point is this:  the $\infty$-categories of sheaves $Sh(M)$ contain a lot of objects which are now irrelevant, in that they never arise if  you start on the point and apply only those functors (namely, pullbacks for arbitrary maps and the left and double left adjoints of pullbacks for $L$-proper submersions) used in the above construction of $(LSph^{top})^{-1}$.  For instance, only locally constant sheaves now arise, so we could replace the $\infty$-categories $Sh(M)$ by their full subcategories of locally constant sheaves and still obtain the same $(LSph^{top})^{-1}$.  But we can get by with an even smaller class of sheaves than this, the class of ``unipotent'' sheaves, provided we restrict to a further full subcategory of $CO$, namely that of $L$-compact orbifolds possessing $L$-unipotent duality; and this is what we will do, in fact skipping right over the class of locally constant sheaves.

Maybe the following table will help conceptualize the different categories of sheaves in play here.

\begin{center}
\begin{tabular} { | l | l | }\hline
This class of sheaves on $M$: & Captures exactly: \\ \hline\hline
All sheaves & The topology of $M$\\ \hline
Locally constant sheaves & The homotopy type of $M$ \\ \hline
Unipotent sheaves (of spectra) & The $E_\infty$-algebra $C^*(M)$ \\ \hline
\end{tabular}
\end{center}

Though again, we've decided to skip over the second row and go straight to the third.  Here is the basic definition:

\begin{definition}\label{unidef}
Let $M$ be an orbifold.  We say that an $\mathcal{F}\in Sh(M)$ is unipotent if it lies in the smallest full subcategory of $Sh(M)$ which is stable, closed under colimits, and contains the unit $\mathbf{1}$.

We denote by $u^*\colon Sh^{uni}(M)\subseteq Sh(M)$ the inclusion of the full subcategory of unipotent sheaves.
\end{definition}

Evidently, the class of unipotent sheaves is closed under pullbacks and smash product, and on the point every sheaf is unipotent.

A warning: it is not clear whether in general this definition commutes with further Bousfield localization, i.e.\ whether when $L'$ is a further localization of $L$, the functor $u^*\otimes L'Sp$ is fully faithful with essential image the unipotent subcategory of $Sh(M)\otimes L'Sp$.  However, this does hold when $M$ is $L$-compact, because then (as we see in the subsequent lemma) the right adjoint $u_*$ of $u^*$ preserves colimits, so the property of the unit map $id\rightarrow u_*u^*$ being an equivalence, and hence the property of full fidelity of $u^*$, is preserved under tensor product (by Proposition \ref{cohprop}).  For this reason, we will only ever use Definition \ref{unidef} in the $L$-compact case, though probably a weaker hypothesis suffices.

Here then is the key technical lemma:

\begin{lemma}\label{cpctuni}
Let $M$ be an orbifold, and set $p\colon M\rightarrow *$.  Then:
\begin{enumerate}
\item The functor $p_*$ detects equivalences between unipotent sheaves on $M$;
\item If $M$ is $L$-compact, then the right adjoint $u_*$ of $u^*\colon Sh^{uni}(M)\rightarrow Sh(M)$ preserves colimits.
\end{enumerate}
\end{lemma}

The argument below for the first claim is pulled from a MathOverflow answer of Sam Gunningham.  (It replaces a more opaque proof which only worked when $M$ is $L$-compact.)

\begin{proof}
For the first claim, by passing to cofibers it suffices to show that  $p_*(\mathcal{F})=0$ implies $\mathcal{F}=0$ for $\mathcal{F}$ unipotent.  So let $\mathcal{F}\in Sh^{uni}(M)$ with $p_*(\mathcal{F})=0$.  Then the $LSp$-enriched mapping object $Map(\mathcal{G},\mathcal{F})$ is zero for any $\mathcal{G}\in Sh^{uni}(M)$, since it is zero for $\mathcal{G}=\mathbf{1}$ and the collection of $\mathcal{G}$ for which it is zero is stable and closed under colimits.  Taking $\mathcal{G}=\mathcal{F}$ we conclude that $id_\mathcal{F}=0$ and so $\mathcal{F}=0$, as desired.

For the second claim, we note that $p_*u_*\simeq p_*$, because $u^*p^*\simeq p^*$.  Thus $p_*u_*$ preserves colimits by $L$-compactness of $M$.  But $p_*$ detects equivalences on the image of $u_*$ by the above, so we deduce that $u_*$ also preserves colimits, as required.
\end{proof}

\begin{corollary}\label{uniprod}
Let $M$ and $N$ be $L$-compact orbifolds.  Then the functor
$$Sh^{uni}(M)\otimes Sh^{uni}(N)\overset{u^*_M\otimes u^*_N}{\longrightarrow} Sh(M)\otimes Sh(N)\simeq Sh(M\times N)$$
is fully faithful with essential image $Sh^{uni}(M\times N)$.
\end{corollary}

Thus $M\mapsto Sh^{uni}(M)$ promotes to a symmetric monoidal functor $CO\rightarrow LPr^{op}$, and the inclusion $u^*\colon Sh^{uni}\subseteq Sh$ promotes to a natural transformation of such.

\begin{proof}
The second item of Lemma \ref{cpctuni} lets us identify the counit of $u^*_M\otimes u^*_N$ with the tensor product of those of $u^*_M$ and $u^*_N$, and since a left adjoint is fully faithful if and only if its counit is an equivalence, this verifies the full faithfulness claim.

Then for the essential image claim, since $Sh^{uni}(M)$ and $Sh^{uni}(N)$ are both generated by their units, so is $Sh^{uni}(M)\otimes Sh^{uni}(N)$; thus the essential image is necessarily contained in $Sh^{uni}(M\times N)$.  But on the other hand the essential image is indeed stable, has all colimits, and contains $\mathbf{1}$; so the two must agree, as desired.
\end{proof}

Now we finish defining the terms in the statement of Theorem \ref{unithm}.

\begin{definition}
Let $f\colon M\rightarrow N$ be a submersion between orbifolds.  We say that $f$ is $L$-unipotent if $f_\natural(\mathbf{1})\in Sh^{uni}(N).$
\end{definition}

Equivalently, $f$ is $L$-unipotent when $f_\natural$ sends unipotent sheaves to unipotent sheaves.

The class of $L$-unipotent submersions is closed under composition and arbitary base-change (the latter by Lemma \ref{subadj}).

\begin{definition}\label{unidual}
Let $M$ be an orbifold.  We say that $M$ possesses $L$-unipotent duality if the dualizing sheaf $\mathbb{D}_M\in Sh(M)$ and its smash-inverse (c.f. Lemma \ref{duallemma}) are both unipotent.
\end{definition}

For example, if $M$ is a topological manifold which is parallelizable (in the microbundle sense), then $\mathbb{D}_M$ is a constant sphere, so $M$ possesses $L$-unipotent for any $L$.  For another example, if $M=BG$ for a finite group $G$, then there is a natural equivalence $\Delta_\natural\simeq\Delta_*$, so $\mathbb{D}_M=\mathbf{1}$ and thus (trivially) $M$ also possesses $L$-unipotent duality for any $L$.  Further examples can be produced from these by taking products, since Lemma \ref{shprod} and Proposition \ref{cohprop} imply that $\mathbb{D}_{M\times N}$ identifies with the exterior smash product of $\mathbb{D}_M$ and $\mathbb{D}_N$.  See also Proposition \ref{criterion}.

\begin{lemma}\label{unidouble}
Let $f\colon M\rightarrow N$ be an $L$-unipotent submersion between $L$-compact orbifolds with $L$-unipotent duality.  Then the functor $f_\natural$ and its left adjoint $f^\natural$ (\ref{doubleleft}) both preserve the class of unipotent sheaves.
\end{lemma}

\begin{proof}
The claim about $f_\natural$ follows immediately from the unipotency of $f$.  As for $f^\natural$, since it differs from $f^*$ by smashing with $\mathbb{D}_{M/N}^{-1}$ (\ref{dualprop}), it suffices to show that $\mathbb{D}_{M/N}^{-1}$ is unipotent.  However, by functoriality of $(-)^\natural$ under composition we have that $\mathbb{D}_M^{-1}$ identifies with $f^\natural((\mathbb{D}_N)^{-1})$, which further identifies with $f^*(\mathbb{D}_N^{-1})\wedge\mathbb{D}_{M/N}^{-1}$; thus $\mathbb{D}_{M/N}^{-1}\simeq f^*(\mathbb{D}_N)\wedge\mathbb{D}_M^{-1}$ is indeed unipotent by the unipotency hypotheses on $M$ and $N$, as required.
\end{proof}

Now, using Corollary \ref{uniprod} and Lemma \ref{unidouble}, we can give the next step of the proof of Theorem \ref{unithm}.

\subsubsection{Second step of the proof of Theorem \ref{unithm}}\label{2ndstep}

It is a further reinterpretation of $(LSph^{top})^{-1}\colon K(CO_*)\rightarrow Pic(LSp)$, or rather of its restriction to $K(CUO_*)$, where $CUO_*$ is the $\infty$-category with fibrations consisting of the $L$-compact orbifolds possessing $L$-unipotent duality, with fibrations the $L$-unipotent submersions.  (It is clear from the definitions that $CUO_*$ is indeed an $\infty$-category with fibrations; incidentally, this is the first claim of Theorem \ref{unithm}.)

This reinterpretation is obtained as follows:  we exactly replicate the construction of $(LSph^{top})^{-1}$ given in the first step of the proof of Theorem \ref{unithm} (see Section \ref{firststep}), except that we replace the $\infty$-categories $Sh$ of sheaves by their full subcategories $Sh^{uni}$ of unipotent sheaves.  All the required formal properties for the construction are still satisfied with this replacement: Corollary \ref{uniprod} gives the symmetric monoidal functoriality of $Sh^{uni}\colon CUO\rightarrow LPr^{op}$, and Lemma \ref{unidouble} immediately implies that Lemmas \ref{subadj} and \ref{doubleleft} (on left adjointability and double-left adjointability of pullbacks by submersions) hold just as well with $Sh^{uni}$ replacing $Sh$ once we're in the $L$-unipotent case.

On the other hand, the construction for $Sh^{uni}$ sits inside the construction for $Sh$ the whole way along via the inclusion $u^*\colon Sh^{uni}\subseteq Sh$; and since in the last step of the construction we restrict to sheaves on a point, where $u^*$ is an equivalence, the two constructions do produce equivalent maps $(LSph^{top})^{-1}\colon K(CUO_*)\rightarrow Pic(LSp)$, as claimed.

This concludes the second step of the proof, which, to summarize, shows that $(LSph^{top})^{-1}\colon K(CUO_*)\rightarrow Pic(LSp)$ comes from the symmetric monoidal functor $Sh^{uni}\colon CUO\rightarrow LPr^{op}$ in the same way that $(Sph^{alg})^{-1}\colon K(PD_*^{op})\rightarrow Pic(LSp)$ was shown to come from the symmetric monoidal functor $Mod\colon PD^{op}\rightarrow LPr^{op}$ at the end of Section \ref{firststep}.

Thus, to finish the proof, it suffices to identify $Sh^{uni}(-)$ with $Mod_{C^*(-)}$.  This will be done in the next and final section.

\subsubsection{The relationship between unipotent sheaves and modules over cochains, and the end of the proof of Theorem \ref{unithm}}

The following proposition lets us recognize which $\infty$-categories are equivalent to modules over some $E_\infty$-algebra in $LSp$, and with functoriality to boot.

\begin{proposition}\label{recognize}
The functor $A\mapsto Mod_A$ from the $\infty$-category of $E_\infty$-algebras in $LSp$ to the $\infty$-category of $E_\infty$-algebras in $LPr$ is fully faithful, and its essential image consists of those $\mathcal{A}\in E_\infty Alg(LPr)$ such that the functor
$$Map(\mathbf{1},-)\colon \mathcal{A}\rightarrow LSp$$
($LSp$-enriched maps out of the unit object of $\mathcal{A}$) preserves colimits and detects equivalences.
\end{proposition}

\begin{proof}
When $L=id$ this is \cite{ha} Prop.\ 7.1.2.7; the case of general $L$ is just the same.
\end{proof}

Evidently, the inverse functor is given by sending such an $\mathcal{A}\in E_\infty Alg(LPr)$ to the $E_\infty$-algebra $End(\mathbf{1})$, meaning the $LSp$-enriched endomorphisms of the unit object of $\mathcal{A}$, with $E_\infty$-structure coming from the symmetric monoidal structure on $\mathcal{A}$.  Note that when $\mathcal{A}=Sh^{uni}(M)$, this endomorphism algebra agrees with the $E_\infty$-algebra $C^*(M)=\underline{Map}(p_\natural\mathbf{1},\mathbf{1})$ appearing in the statement of Theorem \ref{unithm}.  Using this remark, we deduce:

\begin{lemma}\label{translation}
There is an equivalence of symmetric monoidal functors $CO\rightarrow LPr^{op}$ between $M\mapsto Sh^{uni}(M)$ and $M\mapsto Mod_{C^*(M)}$.
\end{lemma}
\begin{proof}
The first claim of Lemma \ref{cpctuni} shows that the categories $Sh^{uni}(M)$ with their smash product symmetric monoidal structure are in the essential image of the fully faithful functor $A\mapsto Mod_A$ of Proposition \ref{recognize}.  This yields the equivalent claim that there is an equivalence of functors $CO\rightarrow E_\infty Alg(LPr)^{op}$ between $Sh^{uni}(-)$ and $Mod_{C^*(-)}$.\end{proof}

Lemma \ref{translation} essentially finishes the proof of Theorem \ref{unithm}, since it identifies the construction of $(LSph^{top})^{-1}\colon K(CUO_*)\rightarrow Pic(LSp)$ given in Section \ref{2ndstep} with the construction of $(Sph^{alg})^{-1}(C^*(-))\colon K(CUO_*)\rightarrow Pic(LSp)$ given at the end of Section \ref{firststep}.  The only thing that's missing is the second claim of Theorem \ref{unithm}, which states that the contravariant functor $C^*\colon CUO\rightarrow E_\infty Alg(LSp)$ sends $L$-unipotent submersions to Poincar\'{e} duality maps and sends pullbacks by $L$-unipotent submersions to pushouts.  This is really just incidental, but we prove it anyway, as an immediate consequence of the following lemma:

\begin{lemma}
\begin{enumerate}
\item Let $f\colon A\rightarrow B$ be a map of $E_\infty$-algebras in $LSp$.  Suppose that $f^*\colon Mod_A\rightarrow Mod_B$ is doubly left adjointable, and that the evaluation of its double left adjoint on the unit object is an invertible object of $Mod_B$.  Then $f$ is a Poincar\'{e} duality map.
\item Let $\sigma$ be a commutative square of $E_\infty$-algebras in $LSp$ such that the square $Mod_\sigma$ in $LPr$ is (say) vertically left adjointable.  Then $\sigma$ is a pushout.
\end{enumerate}
\end{lemma}

\begin{proof} In terms of the Morita theory of \cite{ha} Prop.\ 7.1.2.4 (or rather its $L$-local analog, proved in the same way), the functor $f^*$ corresponds to $B$ as an $A\text{-} B$-bimodule; and its left adjoint, when colimit-preserving, corresponds to the $A$-dual of $B$, viewed as a $B\text{-}A$-bimodule; and the left adjoint of that, when colimit-preserving, corresponds to the $B$-dual of the $A$-dual of $B$, viewed as an $A\text{-}B$-bimodule.  This makes the first assertion plain.

For the second assertion, we note that vertical left adjointability implies horizontal right adjointability by passing to right adjoints; but then evaluating the right adjointed square on the unit object exactly shows that the map from the pushout of the initial part of the square to its terminal object is an equivalence, since pushouts are computed by relative tensor product in this setting.\end{proof}

We finish the appendix with two supplementary results.  The first gives a concrete criterion for unipotency in certain cases:

\begin{proposition}\label{criterion}
Assume our Bousfield localization is $L_{HR}$ for some (discrete) ring $R$.  Let $M$ be an $L_{HR}$-compact orbifold such that $\pi_0(|M|)$ is finite, and let $\mathcal{F}\in Sh(M)$ be a sheaf on $M$ which is locally constant.

Suppose that for every point $x\colon *\rightarrow M$, the $R$-homology groups of the fiber $x^*\mathcal{F}\in L_{HR}Sp$ are zero outside finitely many degrees, and in all degrees have trivial $\pi_1(|M|,x)$-action up to a finite filtration.

Then $\mathcal{F}$ is unipotent.
\end{proposition}

To explain the $\pi_1(|M|,x)$-action on the $R$-homology of $x^*\mathcal{F}$ used in the above statement, recall from \cite{ha} Thm.\ A.1.5 (or rather its $L_{HR}$-analog, which, since the functor $\psi^*$ of loc.\ cit.\ is left adjointable, follows from \cite{ha} Thm.\ A.1.5 and Proposition \ref{cohprop}) that the full subcategory of $Sh(M)$ consisting of the locally constant sheaves identifies with the functor $\infty$-category $Fun(|M|,L_{HR}Sp)$; thus locally constant sheaves are specified by monodromy data.

\begin{proof}
Recall that $u^*\colon Sh^{uni}(M)\rightarrow Sh(M)$ denotes the inclusion of the unipotent sheaves, and that $u_*$ denotes the right adjoint of $u^*$.  Thus we need to show that $u^*u_*\mathcal{F}\rightarrow\mathcal{F}$ is an equivalence.  Since we are working $HR$-locally, we can check this after smashing with $HR$, or equivalently after tensoring our $\infty$-categories with $Mod_{HR}$.  Then since $u_*$ preserves colimits (Lemma \ref{cpctuni}, claim 2), the counit $u^*u_*\rightarrow id$ tensored with $Mod_{HR}$ identifies with the analogous counit for sheaves of $HR$-modules; thus, letting $\mathcal{L}=\mathcal{F}\wedge HR$ and $X=|M|$, we are reduced to showing that if $\mathcal{L}\in Fun(X,Mod_{HR})$ is a local system of $HR$-modules on a space $X\in \mathcal{S}$ with $\pi_0(X)$ finite --- or let's say $X$ connected for simplicity --- and the $R$-homology groups of $\mathcal{L}$ vanish in all but finitely many degrees and in all degrees have trivial $\pi_{\leq 1}(X)$-action up to a finite filtration, then $\mathcal{L}$ can be built out of the constant local system $HR$ on $X$ via shifts and colimits.

However, by induction on the number of nonvanishing homology groups we can reduce to where $\mathcal{L}$ is concentrated in a single degree, and thus up to shifts is represented by an $R[\pi_{\leq 1}(X)]$-module $L$; and then since an extension of one $R[\pi_{\leq 1}(X)]$-module $L_0$ by another $L_1$ can be built as the cofiber of a map $L_0[-1]\rightarrow L_1$ of corresponding local systems of $HR$-modules, we again reduce by induction to where the $\pi_{\leq 1}(X)$-action on $L$ is trivial; but then $L$ is pulled back from the point, and hence is certainly built from $HR$ via desuspensions and colimits, since every $HR$-module on the point is so.
\end{proof}

Finally, here is a result which unwinds some of the above proof of Theorem \ref{unithm} into a concrete statement.

\begin{proposition}\label{algtriv}
Let $Triv_*^\sim$ denote the $E_\infty$-space of $L$-compact pointed orbifolds $M$ with $L$-unipotent duality such that the unit map $\mathbf{1}\rightarrow C^*(M)$ is an equivalence, with $E_\infty$-structure of cartesian product.

Then for $*\overset{e}{\longrightarrow}M\overset{p}{\longrightarrow}*$ in $Triv_*^\sim$ the composition of natural maps $p_*\rightarrow e^*\rightarrow p_\natural\colon Sh(M)\rightarrow Sh(*)$ is an equivalence, and the following two trivializations of $LSph^{top}\colon Triv_*^\sim\rightarrow Inv(LSp)$ agree:
\begin{enumerate}
\item First, the one determined by $LSph^{top}(M)\simeq Sph^{alg}(C^*(M))\simeq Sph^{alg}(\mathbf{1})\simeq \mathbf{1}$, where the first equivalence comes from Theorem \ref{unithm};
\item Second, the one determined by $LSph^{top}(M)=p_\natural e_*(\mathbf{1})\simeq p_*e_*(\mathbf{1})\simeq \mathbf{1},$ where we use the inverse to the above map $p_*\rightarrow e^*\rightarrow p_\natural$ in the first equivalence.
\end{enumerate}
\end{proposition}

\begin{proof}
Note that, by Lemma \ref{translation} and the hypothesis on $C^*(M)$, the functor $p^*\colon Sh(*)\rightarrow Sh^{uni}(M)$ is an equivalence.  In particular, $p^*$ canonically identifies with its double left adjoint $p^\natural$, and the the first trivialization of $LSph^{top}(M)$ given above is obtained from the resulting equivalence $id\simeq e^*p^* \simeq e^*p^\natural$ by passing to right adjoints.  Thus what has to be checked is that the natural transformation $p_*\rightarrow e^*\rightarrow p_\natural$ is right adjoint to this canonically defined equivalence $p^\natural\simeq p^*$.

However, since these left adjoints $p^\natural$ and $p^*$ land in $Sh^{uni}(M)$, they are also adjoint to the unipotent restrictions of $p_\natural$ and $p_*$; thus we see that it suffices to check the adjointness of $p_*\rightarrow e^*\rightarrow p_\natural$ and $p^\natural\simeq p^*$ after restricting the former to $Sh^{uni}(M)$.  Transporting via $p^*\colon Sh(*)\simeq Sh^{uni}(M)$, it therefore suffices to check that $p_*p^*\rightarrow e^*p^*\rightarrow p_\natural p^*$ is right adjoint to $p_\natural p^\natural\simeq p_\natural p^*$.  However, $p^*$ being fully faithful, all these functors are compatibly identified with the identity (via, respectively, the unit $id\rightarrow p_*p^*$, the equivalence $e^*p^*\simeq id$, the counit $p_\natural p^*\rightarrow id$, the unit $id\rightarrow p_\natural p^\natural$, and the counit $p_\natural p^*\rightarrow id$).
\end{proof}

\end{document}